\documentclass[letterpaper, 10 pt, conference]{ieeeconf} 
\usepackage[utf8]{inputenc} 
 
\usepackage{algorithm}
\usepackage{algpseudocode}
\usepackage{amsmath,amssymb,amsthm}
\usepackage{graphicx}
\makeatletter
\providecommand\caption@documentclass{standard}
\makeatother
\usepackage{subcaption}
\usepackage{booktabs}
\usepackage{placeins}
\usepackage{hyperref}
\usepackage{bookmark}
\usepackage{multirow}
\usepackage{caption}
\usepackage{todonotes}
\usepackage{pifont}
\usepackage{mathtools}
\let\labelindent\relax
\usepackage{enumitem}
\setlist[itemize]{leftmargin=*}
\setlist[enumerate]{leftmargin=*}

\newtheorem{assumption}{Assumption}
\newtheorem{definition}{Definition}
\newtheorem{problem}{Problem}
\newtheorem{theorem}{Theorem}
\newtheorem{lemma}{Lemma}
\newtheorem{proposition}{Proposition}
\newtheorem{remark}{Remark}
\makeatletter
\def\endfigure{\end@float}
\makeatother

\usepackage{setspace}
\usepackage{ifthen}
\newboolean{showcomments}
\setboolean{showcomments}{true}
\usepackage{color}
\usepackage{todonotes}

\definecolor{bleudefrance}{rgb}{0.19, 0.55, 0.91}
\definecolor{ao(english)}{rgb}{0.0, 0.5, 0.0}

\newcommand{\addcite}[0]{\ifthenelse{\boolean{showcomments}}
{\textcolor{purple}{(add cite(s)) }}{}}%

\newcommand{\addcites}[0]{\ifthenelse{\boolean{showcomments}}
{\textcolor{purple}{(add cite(s)) }}{}}%

\newcommand{\addref}[0]{\ifthenelse{\boolean{showcomments}}
{\textcolor{purple}{(add ref) }}{}}%

\newcommand{\enrique}[1]{  \ifthenelse{\boolean{showcomments}}
{\todo[inline,color=bleudefrance]{Enrique: #1}}{}}
\newcommand{\jixian}[1]{  \ifthenelse{\boolean{showcomments}}
{\todo[inline,color=lightgray]{Jixian: #1}}{}}
\newcommand{\zhuo}[1]{\ifthenelse{\boolean{showcomments}}
{\todo[inline,color=olive]{Zhuo: #1}}{}}

\newcommand{\hl}[1]{\ifthenelse{\boolean{showcomments}}
{\textcolor{red}{#1}}{#1}}

\newboolean{showedits}
\setboolean{showedits}{false}
\usepackage[markup=underlined]{changes}
\definechangesauthor[color=bleudefrance]{EM}
\newcommand{\aem}[1]{
\ifthenelse{\boolean{showedits}}
{\added[id=EM]{#1}}
{\!#1\hspace{-4.75pt}}
}
\newcommand{\repem}[2]{
\ifthenelse{\boolean{showedits}}
{\replaced[id=EM]{#1}{#2}}
{\!#1\hspace{-4.75pt}}
}
\newcommand{\dem}[1]{
\ifthenelse{\boolean{showedits}}
{\deleted[id=EM]{#1}}
{}
}

\usepackage{readprov}
\ReadFileInfos{ieee.bbx}
\IfFileAtLeastTF{ieee.bbx}{2025-03-11}{%
  \IfFileAtLeastTF{ieee.bbx}{2025-08-01}{}{%
    \expandafter\providecommand\csname abx@macro@volume+number\endcsname{}}%
}{}
\usepackage[backend=biber, style=ieee, doi=false, url=false, eprint=false]{biblatex}
\newcommand{\K}{\mathcal{K}}
\newcommand{\X}{\mathcal{X}}

\newcommand{\R}{\mathbb{R}}

\newcommand{\tgt}{\mathrm{tgt}}
\newcommand{\supp}{\mathrm{supp}}
\newcommand{\Supp}{\mathrm{Supp}}
\newcommand{\cB}{\mathcal{B}}
\newcommand{\cA}{\mathcal{A}}
\newcommand{\cD}{\mathcal{D}}

\usepackage{amsmath}
\newcommand\restr[2]{{
  \left.\kern-\nulldelimiterspace
  #1
  \littletaller
  \right|_{#2} 
  }}
\newcommand{\littletaller}{\mathchoice{\vphantom{\big|}}{}{}{}}

\IEEEoverridecommandlockouts
\title{\LARGE \bf 
Symplectic Inductive Bias for Data-Driven Target Reachability in Hamiltonian Systems
}

\author{Zhuo Ouyang, Jixian Liu, and Enrique Mallada
\thanks{Zhuo Ouyang is with the College of Engineering, Peking University, BJ 100091, P.R.C. {\tt\small 2200011199@stu.pku.edu.cn}. J. Liu and E. Mallada are with the Department of Electrical and Computer Engineering, Johns Hopkins University, MD 21218, U.S.A. 
        {\tt\small jliu376@jh.edu, mallada@jhu.edu}.}
\thanks{This work was supported by the NSF Global Centers program under Grant No.~2330450 and by the DOE Office of Science (ASCR) under Award No.~826565.
}}

\begin{document}
\maketitle
\thispagestyle{empty}
\pagestyle{empty}

\maketitle
\thispagestyle{empty}
\pagestyle{empty}

\begin{abstract}
Inductive bias refers to restrictions on the hypothesis class that enable a learning method to generalize effectively from limited data. A canonical example in control is linearity, which underpins low sample-complexity guarantees for stabilization and optimal control. For general nonlinear dynamics, by contrast, guarantees often rely on smoothness assumptions (e.g., Lipschitz continuity) which, when combined with covering arguments, can lead to data requirements that grow exponentially with the ambient dimension. In this paper we argue that data-efficient nonlinear control demands exploiting inductive bias embedded in nature itself—namely, structure imposed by physical laws. Focusing on Hamiltonian systems, we leverage symplectic geometry and intrinsic recurrence on energy level sets to solve target reachability problems. Our approach combines the recurrence property with a recently proposed class of policies, called chain policies, which composes locally certified trajectory segments extracted from demonstrations to achieve target reachability. We provide sufficient conditions for reachability under this construction and show that the resulting data requirements depend on explicit geometric and recurrence properties of the Hamiltonian rather than the state dimension.
\end{abstract}
\section{Introduction} \label{sec:intro}

A central challenge in data-driven control is how to achieve reliable generalization from limited data. In learning theory, this capability is governed by inductive biases, namely, structural restrictions on the hypothesis class—the set of models or control laws considered—that enable generalization~\cite{vapnik1998statistical,shalev2014understanding,bishop2006pattern}. By limiting the class complexity, inductive bias determines how solutions inferred from finite data extend beyond observed data, and how data requirements scale with problem complexity. In control, such inductive bias typically appears through assumptions on the system model, such as linearity, polynomial structure, or smoothness (e.g., Lipschitz continuity), which define the class of admissible dynamics or policies over which guarantees are derived.

For linear systems, this structure leads to tractable and often sample-efficient data-driven control methods. A broad class of problems—including system identification, stabilization, and optimal control—can be addressed directly from data using a range of techniques, such as convex optimization, LMI-based formulations, and trajectory-based predictive control~\cite{oymak2018nonasymptotic,zheng2021nonasymptotic,hu2022sample,werner2024sample,toso2025learning,dean2020sample,depersis2019formulas,coulson2019data,berberich2020data}. In several cases, these approaches admit rigorous finite-sample guarantees that explicitly characterize how data requirements scale with system dimension and control objectives~\cite{oymak2018nonasymptotic,zheng2021nonasymptotic,hu2022sample,werner2024sample,dean2020sample}. As a result, the linear setting provides a relatively rich understanding of the interplay between data, computation, and control objectives.

In contrast, for nonlinear systems, there is no a priori canonical inductive bias, and different modeling assumptions—such as polynomial, rational, or Lipschitz models—lead to significantly different methodologies and outcomes~\cite{dai2020semi,guo2021data,strasser2021data,monshizadeh2025versatile}. While these approaches enable controller synthesis with explicit guarantees, sample complexity bounds remain comparatively scarce. A few exceptions include sample complexity results for stability~\cite{boffi2020learning}, stabilizability~\cite{siegelmann2025data}, and reachability analysis~\cite{lew2022simple}. Notably, in all these cases, the resulting data requirements scale \textit{exponentially with the state dimension}, reflecting the intrinsic difficulty of data-driven control in nonlinear systems.

The above-mentioned results suggest a fundamental limitation in data-driven control for nonlinear systems. In this paper, we argue that this apparent limitation stems from the choice of function class—such as polynomial or Lipschitz models—which fail to capture the rich structure embedded in physical systems. To overcome this, we advocate for an inductive bias grounded in physical laws. In particular, we focus on Hamiltonian systems, where the dynamics are governed by an energy function and evolve on invariant energy level sets. Leveraging their symplectic structure and intrinsic recurrence properties, we study the problem of target reachability and show how these features can be used to design data-driven control policies that fundamentally alter the dependence of sample complexity on system dimension.

Our approach builds on a novel class control policies, called  chain policies~\cite{siegelmann2025data}, which construct control strategies by composing locally validated control segments derived from data. By recurrently composing the execution control segments, chain policies guarantee certain recurrent conditions that is sufficient for stability and safety in nonlinear systems~\cite{siegelmann2023recurrence,sibai2026recurrence,liu2025recurrent,liu2025safety}. In the Hamiltonian setting, however, recurrence is not imposed as a design principle but arises intrinsically from the dynamics on invariant energy layers, creating natural opportunities to reuse locally valid behaviors. This leads to a reachability framework where global performance can be achieved from a finite collection of trajectory segments. We establish sufficient conditions under which target reachability is guaranteed and show that the associated data requirements depend on intrinsic geometric and recurrence properties of the Hamiltonian—such as energy variation and ergodic structure—rather than the ambient state dimension.

The remainder of the paper is organized as follows. Section~\ref{sec:preliminaries} introduces the Hamiltonian system model and formulates the target reachability problem. Section~\ref{sec:Reachability} presents the main theoretical results, including reachability guarantees and sample complexity bounds under chain policies. Section~\ref{sec:numerical_simu} provides numerical validation on representative systems, and Section~\ref{sec:conclusions_and_future_work} concludes with a discussion of future directions.

\vspace{1ex}
\noindent\emph{Notation:}
We denote by $\|\cdot\|$ the Euclidean norm. For $x \in \mathbb{R}^n$ and $r>0$, let $\cB_r(x) := \{ w \in \mathbb{R}^n \mid \|x-w\| \le r \}$ be the closed ball of radius $r$ centered at $x$. For a compact set $S\subseteq \mathbb{R}^n$, let $D_S:=\sup_{x,y\in S}\|x-y\|$ denote its diameter, $\overline{S}$ its closure, $\partial S$ its boundary, and $\operatorname{int}(S)$ its interior. For  $\forall x\in\mathbb{R}^n$, the distance from $x$ to set $S$ is $\mathrm{d}(x,S):=\inf_{y\in S}\|x-y\|$. $\lceil \cdot \rceil$ and $\lfloor \cdot \rfloor$ denote the ceiling and floor operators, respectively.

\section{Preliminaries and Problem Formulations}
\label{sec:preliminaries}
\subsection{Hamiltonian System}
In this paper, we firstly review mathematical formulation of \textbf{Hamiltonian System}.

\begin{definition}[Hamiltonian System]
A Hamiltonian system without energy dissipation is a dynamical system of the form
\begin{align}
\label{eq:Hamil_sys}
    \dot{x} = f(x,u) = J(x)\nabla H(x) + G(x)u,
\end{align}
where $x \in \X \subset \R^n$ is the state defined on a compact set $\X$, $H:\R^n \to \R$ is the Hamiltonian function, $J(x)\in\R^{n\times n}$ is skew-symmetric, and $u\in U\subset \R^m$ is the input of the system defined on a compact U.
\end{definition}

Physically, $H(x)$ describes the energy of system~\eqref{eq:Hamil_sys}, i.e., sum of kinetic and potential energy in mechanical system, while the skew-symmetric matrix $J(x)$ encodes the intrinsic power-conserving structure. In particular, $J(x)$ determines how energy flows between state variables without creating or dissipating energy. 
We make the following assumptions about system \eqref{eq:Hamil_sys}.

\begin{assumption}[Bounded Hamiltonian Gradient]
\label{ass:bounded_gradient}
The Hamiltonian function $H$ is continuously differentiable on $\X$ and there exists an upper bound $L_H > 0$ such that $\forall x\in\X, \|\nabla H(x)\|\le L_H$.
\end{assumption}

\begin{assumption}[Lipschitz Continuity of the Dynamics]
\label{ass:lipschitz_dynamics}
The dynamics $f(x,u)$ in~\eqref{eq:Hamil_sys} are Lipschitz continuous w.r.t. $x$ in both the state and the input. Namely, for any $u \in U$, there exist constants $L>0$ such that
\begin{align*}
     \|f(x_1,u)-f(x_2,u)\|
    \le
    L\|x_1-x_2\|, \notag \forall x_1,x_2\in\X.
\end{align*}
\end{assumption}

\begin{remark}[Bounded Vector Field]
Since $\X$ and $U$ are compact, and $f$ is continuous by Assumption~\ref{ass:lipschitz_dynamics}, there exists a constant $C_f>0$ such that
\begin{align}
\label{rem:bounded_vector_field}
\|f(x,u)\|\le C_f, \forall x\in\X,\ \forall u\in U.
\end{align}
\end{remark}

Since the Hamiltonian $H(\cdot)$ is conserved along the zero-input dynamics, it is natural to partition the state space into invariant energy layers.

\begin{definition}[Energy Layer]
For each energy value $E \in H(\mathcal X)$, the corresponding energy layer is defined as
\begin{align*}
\Sigma_E := \{x \in \mathcal X : H(x)=E\}.
\end{align*}
\end{definition}

$\Sigma_E$ is called an invariant energy layer because, under the zero-input, every trajectory starting in $\Sigma_E$ remains in $\Sigma_E$ for all future times.

\subsection{Target Reachability Problem}

We now formalize the control objective considered in this paper. Given the Hamiltonian structure introduced above, our goal is to design control inputs that steer the system from a set of admissible initial conditions to a desired target set.

Let $\mathcal{U}^{(0,t]}$ denote the set of admissible control signals on $(0,t]$, where each $u:(0,t]\to U$ is piecewise continuous (and measurable); we also use $\mathcal{U}:=\mathcal{U}^{(0,\infty)}$.
Further, given any initial condition $x\in\mathcal{X}$ and control input $u\in \mathcal{U}^{(0,t]}$, we use $\phi(t,x,u)$ to denote the state of system~\eqref{eq:Hamil_sys} at time $t$.

Let $S_0 \subseteq \mathcal{X}$ be a compact set of admissible initial states, and let $S_{\tgt} \subseteq \mathcal{X}$ be the prescribed target set.

\begin{problem}[Target Reachability]\label{prob:target reachability}
Given system~\eqref{eq:Hamil_sys}, an initial state $x_0 \in S_0$, and a target set $S_{\tgt}$, determine a control signal $u \in \mathcal{U}^{(0,t]}$ and a time $t>0$ such that
\[
\phi(t,x_0,u) \in S_{\tgt}.
\]
\end{problem}

\subsection{Recurrence on Energy Layers}
\label{sec:recurrence_on_energy_layers}

We next recall the recurrence structure of the zero-input dynamics associated with~\eqref{eq:Hamil_sys}. Since the system is lossless, the Hamiltonian is conserved along zero-input trajectories. Hence, for each energy value $E\in H(\X)$, the energy layer $\Sigma_E:=\{x\in\X:H(x)=E\}$
is invariant under the zero-input flow $\phi(t,x,0)$. Our interest is in how trajectories repeatedly revisit dynamically relevant regions on each compact invariant energy layer. We first introduce invariant measures, which describe measures preserved by the zero-input flow.

\begin{definition}[Invariant Measure]
A probability measure $\mu$ on $M$ is said to be invariant under the flow $\phi$ if for any measurable set $A \subseteq M$ and any $t \ge 0$, $\mu(\phi(t,A)) = \mu(A).$
\end{definition}

To further characterize whether trajectories explore the whole invariant set or remain confined to smaller invariant subsets, we next introduce ergodicity.

\begin{definition}[Ergodic Measure]
Let $\mu$ be an invariant probability measure on a set $M$. The measure $\mu$ is said to be ergodic if for any measurable set $A \subseteq M$ that is invariant under the flow, i.e., $\phi(t,A,0) \subseteq A$ for all $t \ge 0$, it holds that $\mu(A) \in \{0,1\}$.
\end{definition}

Intuitively, ergodicity means that trajectories are not confined to smaller invariant subsets, but instead propagate throughout $M$. In particular, for almost every initial condition, trajectories are dense in the support of $\mu$.
The next theorem~\cite[Theorem 5.1.3]{viana2016foundations} illustrates how every finite invariant measure admits an ergodic decomposition.

\begin{theorem}[Ergodic Decomposition on an Energy Layer]
\label{thm_ergodic}
Let $\mu_E$ be an invariant measure on $\Sigma_E$. Then there exists a measurable family of ergodic measures $\{\mu_\alpha^E\}_{\alpha\in\cA_E}$ and a probability measure $\nu_E$ on the index set $\cA_E$ such that
$\mu_E = \int_{\cA_E} \mu_\alpha^E\, d\nu_E(\alpha)$.
\end{theorem}
For each $\alpha\in\cA_E$, define the corresponding support of the ergodic component by
\begin{align*}
K_\alpha^E := \supp(\mu_\alpha^E)\subseteq \Sigma_E,
\end{align*}
where $\supp(\mu):=\{x\in \Sigma_E:\mu(\cB_r(x))>0,\ \forall r>0\}$ is the smallest closed subset of $\Sigma_E$ that has full $\mu$-measure. In Theorem~\ref{thm_ergodic}, each measure $\mu_\alpha^E$ represents an ergodic component of the invariant dynamics on $\Sigma_E$, and $K_\alpha^E$ is the closed region where the corresponding ergodic dynamics take place. On each $K_\alpha^E$, typical trajectories are dense~\cite[Proposition 4.3.5]{viana2016foundations}.

\begin{proposition}[Density of Typical Trajectories in an Ergodic Support]
\label{proposition_density}
For $\mu_\alpha^E$-almost every $x\in K_\alpha^E$, the forward orbit of $x$ is dense in $K_\alpha^E$, namely
\begin{align*}
\overline{\{\phi(t,x,0):t\ge 0\}}=K_\alpha^E.
\end{align*}
Hence, for $\mu_\alpha^E$-almost every initial condition in $K_\alpha^E$, the zero-input trajectory visits every neighborhood of every point in $K_\alpha^E$ infinitely often.
\end{proposition}

\subsection{Chain Policies}
\label{sec:def_chain_policy}
Motivated by the nonparametric chain-policy idea in~\cite{siegelmann2025data},
we now specialize the policy construction to the reachability problem for the system~\eqref{eq:Hamil_sys}. The basic idea is to build a finite library of demonstrated control snippets and then select among them according to the current state.

Suppose we are given a finite set of expert demonstrations $\cD:=\{(x_j,u_j(\cdot),\tau_j)\}_{j=1}^M,$ where each $u_j:(0,\tau_j] \to U$ is a piecewise continuous control signal,
and $x_j$ is the corresponding initial states of the expert demonstrations' trajectories of~\eqref{eq:Hamil_sys}. In addition, let $u_0:(0,\tau_0]\to U$ be a prescribed default control signal,
where $\tau_0>0$.

\begin{definition}[Control Alphabet]
A \textbf{control alphabet} is a finite collection of control signals
\begin{align*}
\cA := \{u_i : (0, \tau_i] \to U\}_{i=0}^{M},
\end{align*}
where each $u_i$ is piecewise continuous and $\tau_i > 0$.
\end{definition}

The control alphabet provides a library of candidate control snippets. To determine where each snippet should be applied in the state space, we introduce an assignment set.

\begin{definition}[Assignment Set]
An \textbf{assignment set} is a finite collection of verification triples
\begin{align*}
\K := \{(x_i, r_i, u_i)\}_{i=1}^{N} \subseteq \R^n \times \mathcal{R}_{> 0} \times \cA,
\end{align*}
where $x_i \in \R^n$ is the center state point, $u_i \in \cA$ is the control signal assigned to that region, and $r_i > 0$ is its effective radius. The support of $\K$ is 
\begin{align*}
\Supp(\K) := \bigcup_{i=1}^{N} \cB_{r_i}(x_i),
\end{align*}
where $N:=\lvert \K \rvert$ is the size of the assignment set.
\end{definition}

While an assignment set specifies regions that the control is effective, it does not by itself resolve which control to apply when balls overlap, nor what to do when a state lies outside $\Supp(\K)$. Based on assignment set, we introduce
a normalized nearest-neighbor selection rule with a default fall-back option. For each $x\in\X$, define $\rho_{\K}(x) := \min_{1\le i\le N}\frac{\|x-x_i\|}{r_i}.$ The associated index map $\iota_{\K}:\X \to \{0,1,\dots,N\}$ is given by
\begin{align*}
\iota_{\K}(x):=
\begin{cases}
\displaystyle \arg\min_{1\le i\le N}\frac{\|x-x_i\|}{r_i},
& \rho_{\K}(x)\le 1,\\[1ex]
0, & \text{otherwise}.
\end{cases}
\end{align*}
Thus, if $x\in\Supp(\K)$, the rule selects the assignment whose normalized
distance is minimal; otherwise, it selects the default control $u_0$. In the
present analysis, we take $u_0$ to be the zero input, so that when the state
lies outside $\Supp(\K)$, the system follows the zero-input dynamics until it
re-enters the support of the assignment set. Building on this rule, we now
formalize the induced nonparametric policy.

\begin{definition}[Nonparametric Chain Policy]
Given an assignment set $\K$ and a default control $u_0$, the nonparametric
chain policy (NCP) is the map
\begin{align*}
    \pi_{\K}:\X\to\cA
\end{align*}
defined by $\pi_{\K}(x)=u_{\iota_{\K}(x)}.$
\end{definition}

\begin{remark}[Execution of the Nonparametric Chain Policy]
\label{rem:excution_of_NCP}
Given an initial state $x_0=x$, the policy $\pi_{\K}$ induces an infinite-horizon control signal by concatenation. For each $n\ge 0$, define recursively $u_n := \pi_{\K}(x_n), T_n := \tau(u_n),
x_{n+1} := \phi(T_n,x_n,u_n)$, where $\tau(u_n)$ denotes the duration of the selected control snippet $u_n$. Let $s_0:=0$ and $s_{n+1}:=s_n+T_n$. Then the induced control signal $u_{\K,x}:(0,\infty)\to U$ is defined by
\begin{align*}
u_{\K,x}(t):=u_n(t-s_n),\qquad t\in[s_n,s_{n+1}).
\end{align*}
\end{remark}

\section{Reachability in Hamiltonian Systems}
\label{sec:Reachability}
We are now ready to present the main results of this paper. We establish target reachability by combining two ingredients: local control actions that reduce an energy-based distance to the target, and recurrence of the zero-input Hamiltonian flow, which returns trajectories to regions where those actions can be reused. Repeating this interplay allows the trajectory to reach the target energy band and, ultimately, the target set. This separation between controlled energy reduction and passive recurrence underlies all results in this section.

\subsection{Target Reachability via Chain Policies}
\label{sec:general reachability}

As mentioned above, our strategy for reachability is energy-based. We aim to design control actions that drive the system toward the energy levels associated with the target set. 
Since $S_{\tgt}$ is compact and connected, its image under the Hamiltonian is an interval,
\[
H(S_{\tgt}) = [H_{\min}, H_{\max}], 
\]
where~$H_{\min}\!\!:=\!\min_{x \in S_{\tgt}}\!\! H(x)$ and $H_{\max}\!\!:=\!\max_{x \in S_{\tgt}}\!\! H(x).$

Once the trajectory reaches this energy band, the zero-input dynamics preserve energy, and the ergodic structure of each energy layer can be used to reach the target set, provided that the target is not dynamically isolated within the layer.
This motivates the following assumption.

\begin{assumption}[Ergodic Component Coverage]
\label{ass_target_set}
For every $E \in [H_{\min}, H_{\max}]$ and every ergodic component $K_\alpha^E \subseteq \Sigma_E$,
\[
S_{\tgt} \cap K_\alpha^E \neq \varnothing.
\]
\end{assumption}

\begin{remark}
    Assumption \ref{ass_target_set} is done for ease of exposition. Violation of this assumption would require a more sophisticated strategy that in philosophy does not depart from the presented here and is left for the journal version of this paper.
\end{remark}

To quantify progress toward the target set, we introduce an energy-based distance that measures how far a state lies from the target energy interval. This quantity will serve as a Barrier-like function that we aim to decrease through control actions.

\begin{definition}[Energy Signed Distance to $S_\mathrm{\tgt}$]
\label{def_energy_distance}
Let $H(S_{\tgt}) = [H_{\min}, H_{\max}]$, and define
\[
H^\star_+ := \frac{H_{\max}+H_{\min}}{2}, 
\qquad 
H^\star_- := \frac{H_{\max}-H_{\min}}{2}.
\]
The energy distance to the target set is defined as
\begin{align*}
\Delta H(x) := |H(x)-H^\star_+| - H^\star_-.
\end{align*}
In particular, $\Delta H(x) \le 0$ if and only if $H(x) \in H(S_\mathrm{\tgt})$.
\end{definition}

We are therefore interested in finding controls that bring the system toward the set
\[
H_{\tgt} := \{x \in \mathcal{X} : \Delta H(x) \leq 0\}.
\]
However, in order to provide guarantees, we will require our demonstrations to reach a slightly smaller set. Thus, for any $0 < \epsilon < H^\star_-$, we define
\[
H_{\tgt}^\epsilon := \{x \in \mathcal{X} : \Delta H(x)\leq -\epsilon\}.
\]
This leads to the following assumption.

\begin{assumption}[Reachability of $H_{\tgt}^\epsilon$]
\label{ass_reachability_of_the_target_set}
For all $x \in \mathcal{X} \setminus H_{\tgt}^\epsilon$, there exist $T > 0$ and $u \in \mathcal{U}^{(0,T]}$ such that
\[
\phi(T,x,u) \in H_{\tgt}^\epsilon.
\]
\end{assumption}

To quantify how quickly the system can be driven toward the target energy band, we introduce the corresponding first hitting time. For any $x \in \mathcal{X}$ and control signal $u \in \mathcal{U}$, define
\begin{align*}
T_\epsilon(x,u) := \inf\{t>0 : \phi(t,x,u)\in H_{\tgt}^\epsilon\}.
\end{align*}
Thus, using the energy distance $\Delta H(x)$, we can compute the average decrease rate as
\begin{align}
v_\epsilon(x,u) := \frac{\Delta H(x)-\Delta H(\phi(T_\epsilon(x,u),x,u))}{T_\epsilon(x,u)}.
\label{eq:energy_rate}
\end{align}
By definition $\phi(T_\epsilon(x,u),x,u)\in \partial H_{\tgt}^\epsilon$, thus
\begin{align*}
v_\epsilon(x,u)=\frac{\Delta H(x)+\epsilon}{T_\epsilon(x,u)}.
\end{align*}
Thus the best achievable decrease rate at $x$ is given by
\begin{align}
v_\epsilon(x):=\sup_{u\in\mathcal{U}} v_\epsilon(x,u)=\frac{\Delta H(x)+\epsilon}{T_\epsilon^\star(x)},
\label{eq:best_rate}
\end{align}
where $T^\star_\epsilon(x)$ is the optimal hitting time maximized \eqref{eq:best_rate}, i.e., $T^\star_\epsilon(x):=\inf_{u\in\mathcal U}T_\epsilon(x,u)$.
This leads to our final requirement.
\begin{assumption}[Uniform Positive Energy Decrease Rate]
\label{ass_lowest_speed}
There exists $\epsilon>0$ such that
\[
\underline{v_\epsilon} := \inf_{x\in \mathcal{X}\setminus H_{\tgt}^\epsilon} v_\epsilon(x) > 0.
\]
\end{assumption}

We will use the above assumptions to ensure uniform energy decrease toward the target energy band and repeated opportunities to apply control through recurrence of the zero-input dynamics. Once the trajectory reaches this energy band, the ergodic structure of the Hamiltonian flow ensures eventual arrival to the target set.

\begin{theorem}[Target Reachability]
\label{thm_general_reachability}
Consider system~\eqref{eq:Hamil_sys} under Assumptions~\ref{ass:bounded_gradient}--\ref{ass_lowest_speed}. 
Let $\mathcal{K} = \{(x_i, r_i, u_i)\}_{i=1}^N$ be a finite assignment set. 
Assume that $\mathcal{K}$ satisfies the following:

\begin{enumerate}
\item \textbf{Local energy decrease:}\label{condition1}
For each $(x_i,r_i,u_i)\in \mathcal{K}$,
\begin{align*}
\Delta H(x_i(\tau_i)) \!+\! v_0 \tau_i \!+\! L_H r_i e^{L \tau_i}
\!\le\! \Delta H(x_i) \!-\! L_H r_i,
\end{align*}
where $x_i(\tau_i):=\phi(\tau_i,x_i,u_i)$, $\tau_{\min}=\min_i\tau_i,\ \text{and }v_0 {>} 0$.
\item \textbf{Energy coverage:}\label{condition2}
Let $c := \sup_{x\in S_0} \Delta H(x)$. Then
\[
\Delta H(x)\le c \;\implies\; H(x)\in H(\Supp(\mathcal{K})).
\]
\item \textbf{Ergodic coverage:}\label{condition3}
For all $E \in H(\Supp(\mathcal{K}))$ and all ergodic components $K_\alpha^E \subseteq \Sigma_E$,
\[
\mathrm{int}(\Supp(\mathcal{K})) \cap K_\alpha^E \neq \emptyset.
\]
\end{enumerate}

Then, the chain policy $\pi_{\mathcal{K}}$ ensures that for almost every $x_0 \in S_0$, there exists $t<\infty$ such that
\[
\phi(t,x_0,\pi_{\mathcal{K}}) \in S_{\tgt}.
\]
\end{theorem}

\begin{proof}
We first establish three key points:  
(i) each control segment decreases the energy distance on its associated support ball,  
(ii) whenever the trajectory leaves the support, the zero-input dynamics return it to the support in finite time for almost every initial condition, and  
(iii) once the trajectory enters the target energy band, it reaches the target set in finite time for almost every initial condition.  
We then combine these three arguments to conclude the theorem.

\medskip
\noindent\textbf{Step 1: Energy decrease on the support.}
Let $y \in \Supp(\mathcal{K})$. Then there exists $i$ such that $y \in \mathcal{B}_{r_i}(x_i)$, i.e., $\|y-x_i\|\le r_i$. By Lipschitz continuity of the dynamics and Grönwall’s inequality,
\[
\|\phi(t,x_i,u_i)-\phi(t,y,u_i)\|\le r_i e^{Lt}.
\]
Since $\|\nabla H(x)\|\le L_H$, the function $\Delta H$ is Lipschitz with constant $L_H$, and therefore
\begin{align}
\label{eq:deltaH_transfer_proof}
\Delta H(\phi(\tau_i,y,u_i))
&\le \Delta H(\phi(\tau_i,x_i,u_i)) + L_H r_i e^{L\tau_i},\\
\label{eq:deltaH_initial_proof}
\Delta H(x_i)
&\le \Delta H(y) + L_H r_i.
\end{align}
Combining \eqref{eq:deltaH_transfer_proof}--\eqref{eq:deltaH_initial_proof} with Condition~\ref{condition1} yields
\begin{align}
\label{eq:DeltaH_decrease_proof}
\Delta H(\phi(\tau_i,y,u_i)) + v_0\tau_i \le \Delta H(y).
\end{align}
Thus, whenever the state lies in $\Supp(\mathcal{K})$, the corresponding control segment strictly decreases the energy distance to the target set.

\medskip
\noindent\textbf{Step 2: Return to the support.}
Suppose that after applying a control segment from some $y\in \Supp(\mathcal{K})$, the state
\[
y':=\phi(\tau_i,y,u_i)
\]
lies outside $\Supp(\mathcal{K})$. From \eqref{eq:DeltaH_decrease_proof}, we have
\[
\Delta H(y') \le \Delta H(y).
\]
Since the trajectory starts from $S_0$ and $\Delta H$ decreases along each controlled execution, it follows that
\[
\Delta H(y') \le c := \sup_{x\in S_0}\Delta H(x).
\]
By Condition~\ref{condition2}, this implies that $H(y')\in H(\Supp(\mathcal{K}))$. Notably, for a point $y'$ s.t. $H(y')\in H(\Supp(\mathcal{K}))$,  the chain policy applies $u\equiv 0$, so the Hamiltonian is preserved and the trajectory remains on the energy layer
\[
\Sigma_E, \quad E:=H(y').
\]
By Condition~\ref{condition3}, $\mathrm{int}(\Supp(\mathcal{K}))$ intersects every ergodic component $K_\alpha^E\subseteq \Sigma_E$. Hence, by Theorem~\ref{thm_ergodic} and Proposition~\ref{proposition_density}, for almost every initial condition $y'\in \Sigma_E$, the zero-input trajectory $\phi(t,y',0)$ is dense in its ergodic component. Therefore, for almost every $y'$ s.t. $H(y')\in H(\Supp(\mathcal{K}))$,  there exists a finite time $T>0$ such that
\[
\phi(T,y',0)\in \Supp(\mathcal{K}).
\]

\medskip
\noindent\textbf{Step 3: Reachability inside the target energy band.}
Suppose now that $y\in H_{\tgt}$, i.e., $H(y)\in [H_{\min},H_{\max}]$. By Assumption~\ref{ass_target_set}, the target set $S_{\tgt}$ intersects every ergodic component of the energy layer $\Sigma_{H(y)}$. Since the zero-input dynamics preserve the Hamiltonian, the trajectory remains on $\Sigma_{H(y)}$. Therefore, by Theorem~\ref{thm_ergodic} and Proposition~\ref{proposition_density}, for almost every initial condition $y\in H_{\tgt}$, the zero-input trajectory is dense in its ergodic component and hence intersects $S_{\tgt}$ in finite time. That is, for almost every $y\in H_{\tgt}$, there exists $T'>0$ such that $\phi(T',y,0)\in S_{\tgt}.$

\medskip
\noindent\textbf{Conclusion.}
Starting from any $x_0 \in S_0$, the chain policy alternates between controlled segments and zero-input evolution. By Step~1, each controlled execution decreases $\Delta H$ by at least $v_0\tau_{\min}>0$, and thus after at most $N_*=\left\lceil \frac{c}{v_0\tau_{\min}} \right\rceil$ executions the trajectory enters $H_{\tgt}$. By Steps~2 and~3, each excursion outside the support returns in finite time, and once in $H_{\tgt}$ the trajectory reaches $S_{\tgt}$ in finite time for almost every initial condition.

Since only finitely many such events occur and the flow maps are continuous, the union of all exceptional null sets remains null after finitely many concatenation of controls. Therefore, for almost every $x_0 \in S_0$, there exists $t<\infty$ such that $\phi(t,x_0,\pi_{\mathcal K}) \in S_{\tgt}.$
\end{proof}

\begin{remark}
Theorem~\ref{thm_general_reachability} shows that reachability does not require coverage of the full state space. Instead, it is sufficient to cover (i) the one-dimensional energy interval connecting $S_0$ to $S_{\tgt}$, and (ii) the ergodic components within each corresponding energy layer.
This reduces the coverage requirement from the full $n$-dimensional state space to a structure parameterized by energy and the ergodic index $\alpha$. In particular, reachability can be achieved by covering a set whose effective dimension is that of $\alpha$ plus one, accounting for energy, without requiring demonstrations throughout the state space.
\end{remark}

\subsection{Existence of the Chain Policy}

Theorem~\ref{thm_general_reachability} provides conditions on the assignment set $\mathcal{K}$ under which the chain policy guarantees target reachability. However, it is not a priori clear whether such conditions can be satisfied using a finite set of control segments. To address this question, we first derive upper and lower bounds on the quantities that govern energy decrease and recurrence, which will allow us to establish existence and sample complexity guarantees for $\mathcal{K}$.

\begin{lemma}[Velocity and Hitting-Time Bounds]
\label{lemma:lower_bound_time}
Under Assumption~\ref{ass_lowest_speed}, the energy decrease rate and the first hitting time to $H_{\tgt}^\epsilon$ satisfy, for all $x \in \mathcal{X}\setminus H_{\tgt}$,
\begin{align*}
v_\epsilon(x) \le L_H C_f, \quad\text{ and }\quad
T_\epsilon^\star(x) \ge \frac{\epsilon}{L_H C_f}.
\end{align*}
\end{lemma}

\begin{proof}
From~\eqref{rem:bounded_vector_field}, we have
\begin{align*}
&|\Delta H(x) - \Delta H(\phi(T,x,u))|
\le L_H \|x - \phi(T,x,u)\| \\
&= L_H \Big\|\int_0^T f(\phi(t,x,u),u(t))\,dt\Big\|
\le L_H C_f T.
\end{align*}
Combining the above with~\eqref{eq:energy_rate}, it follows that
\[
v_\epsilon(x,u) = \frac{\Delta H(x)-\Delta H(\phi(T_\epsilon(x,u),x,u))}{T_\epsilon(x,u)} \le L_H C_f.
\]
Hence $v_\epsilon(x)\le L_H C_f$, and
\[
v_\epsilon(x)=\frac{\Delta H(x)+\epsilon}{T_\epsilon^\star(x)} \le L_H C_f,
\;\; \Rightarrow \;\;
T_\epsilon^\star(x)\ge \frac{\Delta H(x)+\epsilon}{L_H C_f}.
\]
Since $\Delta H(x)>0$ for all $x\in \mathcal{X}\setminus H_{\tgt}$, we obtain
\[
T_\epsilon(x,u)\ge T_\epsilon^\star(x)\ge \frac{\epsilon}{L_H C_f},
\quad \forall x\in\mathcal{X}\setminus H_{\tgt},\ u\in\mathcal{U}.
\]
\end{proof}

To establish the existence of the assignment set $\mathcal{K}$, we reduce the problem to a covering argument. In particular, constructing $\mathcal{K}$ requires (i) controlling the complexity of the dynamics within each energy layer, and (ii) relating spatial coverage of $\mathcal{X}$ to coverage of the corresponding energy values. The following two assumptions address these requirements.

\begin{assumption}[Ergodicity of Energy Layers]
\label{ass_dim_ergodic_components}
For every $E \in H(\Supp(\mathcal{K}))$, the energy layer $\Sigma_E$ is ergodic, i.e.,
\[
\forall E \in H(\Supp(\mathcal{K})),\quad K_\alpha^E = \Sigma_E.
\]
\end{assumption}

\begin{assumption}[Strong Convexity of the Hamiltonian]
\label{ass_strong_convexity}
The Hamiltonian function $H$ is strongly convex on $\mathcal{X}$, i.e., there exists $\mu_H>0$ such that for all $x,y\in\mathcal{X}$,
\[
H(y)\ge H(x)+\nabla H(x)^\top (y-x)+\frac{\mu_H}{2}\|y-x\|^2.
\]
\end{assumption}

\begin{remark}
Assumption~\ref{ass_dim_ergodic_components} reduces each energy layer to a single dynamically connected region, so that the relevant coverage is effectively one-dimensional and parameterized by energy. This corresponds to a simple setting, which includes, for example, Anosov energy surface and Axiom~A systems~\cite{plante1972anosov, hopf1971ergodic,bowen1975ergodic}. In the numerical section, we also consider examples where this assumption is not satisfied. Extending the analysis to more general ergodic decompositions is left for future work.
\end{remark}

\begin{remark}
\label{rem:Strong_Convexity}
Assumption~\ref{ass_dim_ergodic_components} ensures that each energy layer behaves as a single dynamically connected region, eliminating the need to cover multiple ergodic components. Assumption~\ref{ass_strong_convexity} provides a regular relationship between distance in state space and variation in energy, allowing geometric coverings of $\mathcal{X}$ to translate into coverage of the corresponding energy values. Together, these assumptions enable finite coverings that lead to the construction of $\mathcal{K}$.
\end{remark}

\begin{theorem}[Existence of Chain Policy]
\label{thm_existence_chain_policy}
Consider system~\eqref{eq:Hamil_sys} satisfying Assumptions~\ref{ass_reachability_of_the_target_set}--Assumption~\ref{ass_strong_convexity}. Let $v_0 \in (0,\underline{v_\epsilon})$.
Then there exists a nonparametric chain policy $\pi_{\mathcal{K}}$, constructed from a finite assignment set $\mathcal{K}=\{(x_i,r_i,u_i)\}_{i=1}^{N}$, where
\begin{align*}
r_i=\frac{(v_\epsilon(x_i)-v_0)T_\epsilon^\star(x_i)}{L_H(1+e^{LT_\epsilon^\star(x_i)})}>0,
\end{align*}
with $u_i:(0,T_\epsilon^\star(x_i)]\to {U}$ being the optimal control for reaching $\partial H_{\tgt}^\epsilon$ from $x_i$, $T_\epsilon^\star(x_i)$ denoting its hitting time, such that the following holds:
\begin{enumerate}
\item \textbf{Reachability:} For almost every $x_0\in S_0$, there exists $t<\infty$ such that $\phi(t,x_0,\pi_{\mathcal{K}})\in S_{\tgt}$.
\item \textbf{Sample complexity:} Let $H(\Supp(\mathcal{K}))=[H_1,H_2]$. Then
\[
N\leq (H_2-H_1)\frac{16L_H^2}{\mu_H (1-\frac{v_0}{\underline{v_\epsilon}})^2}\frac{\exp(\frac{2L(L_HD_{\mathcal{X}}+\epsilon)}{\underline{v_\epsilon}})}{\epsilon^2}.
\]
\end{enumerate}
\end{theorem}

\begin{proof}
We construct a canonical assignment set and then verify the three conditions of Theorem~\ref{thm_general_reachability}.

\medskip
\noindent\textbf{Step 1: Canonical local construction.}
Let 
\[
\X_c:=\{x\in \X:\Delta H(x)\le c\}\setminus H_{\tgt}^\epsilon.
\]
For every $x \in \X_c$, by Assumptions~\ref{ass_reachability_of_the_target_set}
and~\ref{ass_lowest_speed}, choose an optimal control $u_x:(0,T_\epsilon^\star(x)] \to U$, choose an optimal control $u_x:(0, T_\epsilon^\star(x)] \to U$ such that 
\[
\phi(T_\epsilon^\star(x),x,u_x) \in \partial H_{\tgt}^\epsilon, \,
v_\epsilon(x) = \frac{\Delta H(x) + \epsilon}{T_\epsilon^\star(x)}
\]
Define the certified radius
\[
r(x)
:=
\frac{(v_\epsilon(x)-v_0) T_\epsilon^\star(x)}
{L_H\bigl(1+e^{LT_\epsilon^\star(x)}\bigr)}.
\]
Since $v_0<\underline v_\epsilon\le v_\epsilon(x)$, we have $r(x)>0$.
Moreover,
\begin{align*}
 & \Delta H(\phi(T_\epsilon^\star(x),x,u_x))
+v_0T_\epsilon^\star(x)
+L_Hr(x)e^{LT_\epsilon^\star(x)}\\
= &
\Delta H(x)-L_Hr(x),
\end{align*}
so each triple $(x,r(x),u_x)$ satisfies Condition~\ref{condition1} of
Theorem~\ref{thm_general_reachability}.

\noindent\textbf{Step 2: Uniform lower bound on the certified radii.}
According to Assumption~\ref{ass_lowest_speed}: $v_\epsilon(x)\geq \underline{v_\epsilon}>0,
v_\epsilon(x)=\frac{\Delta H(x_i)+\epsilon}{T_\epsilon^\star(x)},$
we have $T_\epsilon^\star(x) \leq \frac{\Delta H(x)+\epsilon}{\underline{v_\epsilon}}$.
Now recall that $r
=
\frac{\Delta H(x)+\epsilon-v_0 T_\epsilon^\star(x)}
{L_H\bigl(1+e^{LT_\epsilon^\star(x)}\bigr)}$. Then
\[
\Delta H(x)+\epsilon-v_0 T_\epsilon(x)^\star
\geq
(\Delta H(x)+\epsilon)\left(1-\frac{v_0}{\underline{v_\epsilon}}\right)> 0,
\]
because of 
$0<v_0<\underline{v_\epsilon}$.

Therefore, for each fixed $x \in \X_c$, we obtain the lower bound of radius $r(x)$
\begin{align}
r(x) & 
\geq
\frac{
(\Delta H(x)+\epsilon)\left(1-\frac{v_0}{\underline{v_\epsilon}}\right)
}{
L_H\left(
1+\exp\!\left(\frac{L(\Delta H(x)+\epsilon)}{\underline{v_\epsilon}}\right)
\right)
} \notag \\
& \ge \frac{
\epsilon \left(1-\frac{v_0}{\underline{v_\epsilon}}\right)
}{2 L_H \exp(\frac{L(L_H D_{\mathcal{X}} +\epsilon)}{\underline{v_\epsilon}})
}, \label{lower_bound_r_i}
\end{align}
where equation~\eqref{lower_bound_r_i} holds from the fact that $H(x)$ is $L_H$-continuous.

\noindent\textbf{Step 3: Each certified ball covers a nontrivial energy interval.}
Fix $x\in\X_c$ and consider the ball $\cB_{r(x)}(x)$.
By Assumption~\ref{ass_strong_convexity}, for any $\|v\|=1$, we have
\begin{align*}
H(x+rv)+H(x-rv)\ge2H(x)+ \mu_H r^2,
\end{align*}
so at least one of $x\pm rv$ is larger than the average value, that is to say, 
\begin{align}
\label{eq:lower_bound}
\max \{H(x+rv),H(x-rv)\}\geq H(x)+\frac{\mu_H }{2}r^2.
\end{align}
For any $y_1, y_2$ and $y\in \cB_{r(x)}(x)$, according to~\eqref{eq:lower_bound} it follows that
\begin{align*}
& \max_{y_1,y_2\in \cB_{r(x)}(x)} \bigl(H(y_1)-H(y_2)\bigr)\\
& \ge \max_{y\in \cB_{r(x)}(x)} \bigl(H(y)-H(x)\bigr)\\
& \ge \frac{\mu_{H}}{2}r(x)^2.
\end{align*}
If $H_+^\star\notin H(\cB_{r(x)}(x))$, we have: 
\begin{align*}
    &\max_{y_1,y_2\in \cB_{r(x)}(x)} |\Delta H(y_1)-\Delta H(y_2)|\\
    &=\max_{y_1,y_2\in \cB_{r(x)}(x)} |H(y_1)-H(y_2)|\\
    &\geq \frac{\mu_{H}}{2}r(x)^2.
\end{align*}
And if $H_+^\star\in H(\cB_{r(x)}(x))$, since~\eqref{eq:lower_bound} holds, there exists at least one $y^\star \in \cB_{r(x)}(x)$ such that $H(x)+\frac{\mu_{H}}{2}r(x)^2 \le H(y^\star) \le \max H(\cB_{r(x)}(x))$. Thus $[H(x),H(x)+\frac{\mu_{H}}{2}r(x)^2]\subseteq H(\cB_{r(x)}(x))$, and therefore we have:
\begin{align*}
    \max\{|H(x)-H_+^\star|,|H(x)+\frac{\mu_{H}}{2}r(x)^2-H_+^\star|\}\geq \frac{\mu_{H}}{4}r(x)^2.
\end{align*}
Consequently, for any $y_1, y_2 \in \cB_{r(x)}(x)$, we have
\begin{align}
&\max_{y_1,y_2\in B_r(x)} |\Delta H(y_1)-\Delta H(y_2)| \notag\\
& \ge \max_{y_1\in B_r(x)}|H(y_1)-H_+^\star| \label{eq_distance_lower_bound}\\
& \geq \frac{\mu_{H}}{4}r(x)^2,\notag
\end{align}
where inequality~\eqref{eq_distance_lower_bound} holds since we can chose $y_2 \in \cB_{r(x)}(x)$ such that $H(y_2) = H_+^*$. 

Above all, we know that $\forall x \in \X_c$, we have
\begin{align}
& \max_{y_1,y_2\in \cB_{r(x)}(x)} |\Delta H(y_1)- \Delta H(y_2)| \ge \frac{\mu_{H}}{4}r(x)^2 \notag\\
\ge & \frac{\mu_{H} (1-\frac{v_0}{\underline{v_\epsilon}})^2}{16L_H^2}\frac{\epsilon^2}{\exp(\frac{2L(L_HD_{\mathcal{X}}+\epsilon)}{\underline{v_\epsilon}})} \label{eq:energy_interval_lower}
\end{align}

\noindent\textbf{Step 4: Finite covering of the relevant energy interval.}
Let \[
[H_1,H_2]:=H(\X_c)\cup H(H_{\tgt}^\epsilon).
\]
For every $E\in[H_1,H_2]$, choose any $x\in\X_c$ with $H(x)=E$.
Then $E\in H(\cB_{r(x)}(x))$, so the family $\{H(\cB_{r(x)}(x))\}_{x\in\X_c}$ covers $[H_1,H_2]$.
Since $[H_1,H_2]$ is compact and 
\begin{align*}
& \max_{y_1,y_2\in \cB_{r(x)}(x)} |H(y_1)-H(y_2)| \geq \frac{\mu_{H}}{4}r(x)^2\\
\ge & \frac{\mu_{H} (1-\frac{v_0}{\underline{v_\epsilon}})^2}{16L_H^2}\frac{\epsilon^2}{\exp(\frac{2L(L_HD_{\mathcal{X}}+\epsilon)}{\underline{v_\epsilon}})}
\end{align*} there exists a finite subcover
\[
[H_1,H_2]\subseteq \bigcup_{i=1}^N \cB_{r_i}(x_i).
\]
For each selected center $x_i$, define $r_i:=r(x_i), u_i:=u_{x_i},$ and let
\[
\K:=\{(x_i,r_i,u_i)\}_{i=1}^N,
\]
where this finite selection implies
\[
[H_1,H_2]\subseteq H(\Supp(\K)),
\]
which verifies Condition~\ref{condition2} of
Theorem~\ref{thm_general_reachability}.

\noindent\textbf{Step 5: Ergodic Coverage.}
Moreover, under Assumption~\ref{ass_dim_ergodic_components},
each energy layer in the relevant range is itself a unique ergodic component.
Since $\Supp(\K)$ intersects every energy level in $[H_1,H_2]$,
Condition~\ref{condition3} of
Theorem~\ref{thm_general_reachability} also holds.

Therefore all three conditions of
Theorem~\ref{thm_general_reachability} are satisfied, and the resulting
chain policy $\pi_{\K}$ guarantees that for almost every $x_0\in S_0$,
there exists $t<\infty$ such that
\[
\phi(t,x_0,\pi_{\K})\in S_{\tgt}.
\]

\noindent\textbf{Step 6: Sample complexity bound.}
By~\eqref{eq:energy_interval_lower}, every selected interval $\cB_{r_i}(x_i)$
has length at least $\eta$. Hence a greedy interval-covering argument on
$[H_1,H_2]$ gives
\begin{align*}
& N \le \frac{H_2-H_1}{\frac{\mu_{H} (1-\frac{v_0}{\underline{v_\epsilon}})^2}{16L_H^2}\frac{\epsilon^2}{\exp(\frac{2L(L_HD_{\mathcal{X}}+\epsilon)}{\underline{v_\epsilon}})}}\\
& =
(H_2-H_1)\,
\frac{16L_H^2}
{\mu_H\left(1-\frac{v_0}{\underline v_\epsilon}\right)^2}
\frac{
\exp\!\left(
\frac{2L(L_HD_{\X}+\epsilon)}{\underline v_\epsilon}
\right)}
{\epsilon^2}.
\end{align*}
This completes the proof.
\end{proof}

\subsection{Finite Time Reachability}
\label{sec:finite_time}

The previous subsection establishes a general reachability theorem and the existence of a chain policy that realizes it. We now refine this qualitative result into a finite-time guarantee by deriving a uniform upper bound on the time required for the chain policy to drive the system to the target set. The key additional assumption is the uniform bounds on the uncontrolled hitting times:

\begin{assumption}[Upper Bound of Hitting Time]
\label{ass_upper_bound_return_time}
    Assume that there are upper bounds of the return time:
\begin{enumerate}
    \item Return time to the support set $\Supp(\K)$: $\forall x$ satisfies $H(x)\in [H_{1},H_{2}]$, there exists a finite time $T_1 > 0, 
    \mathrm{s.t.}$ $\min_{t\in(0,T_1]}\mathrm{d}(\phi(t,x,0),\Supp(\K))=0$
    \item Reaching time to the target set $S_{\tgt}$: $\forall x$ satisfies $H(x) \in [H_{\min},H_{\max}]$, there exists a finite $ T_2 > 0, \mathrm{s.t.}$ $\min_{t\in(0,T_2]}\mathrm{d}(\phi(t,x,0),S_{\tgt})=0$
\end{enumerate}
\end{assumption}

Under this assumption, the maximal time needed to reach the target set can be bounded explicitly.

\begin{theorem}[Finite-Time Reachability]
Let $\mathcal{K}$ be an assignment set satisfying the conditions of Theorem~\ref{thm_general_reachability}, and define $\tau_{\min}:=\min_i \tau_i$. Under Assumption~\ref{ass_upper_bound_return_time}, the time required to reach the target set from any initial state $x\in S_0$ is uniformly bounded by
\[
T_{\max} \le \frac{L_H D_{\mathcal X}}{v_0}\Bigl(1+\frac{T_1}{\tau_{\min}}\Bigr) + T_2.
\]
\end{theorem}

\begin{proof}
  For any initial state $x_0 \in S_0$, Theorem~\ref{thm_general_reachability} ensures that each time the trajectory leaves the support set $\Supp(\mathcal{K})$, it returns to the support within time at most $T_1$. 

Accordingly, we construct two sequences of states $\{x_i\}_{i=0,\dots N}$ and $\{y_i\}_{i=0,\dots, N}$ as follows. Starting from $x_i$, there exists a time $0 \le t_{1,i} \le T_1$ such that
\[
y_i = \phi(t_{1,i}, x_i, 0) \in \Supp(\mathcal{K}).
\]
From $y_i$, applying the control $u_i\in \mathcal{U}^{(0,\tau_i]}$ for time $t_{2,i}:=\tau_i$ yields
\[
x_{i+1} = \phi(t_{2,i}, y_i, u_i),
\]
with the energy decrease condition
\[
\Delta H(x_{i+1}) + v_0 t_{2,i} \le \Delta H(y_i),
\]
for all $i = 0,\dots,N$.

Thus the policy, iteratively induces the sequence $x_i\xrightarrow {t_{1,i}} y_i\xrightarrow {t_{2,i}} x_{i+1}$. 
    
Now, let $N$ be the smallest index such that $x_N \in H_{\tgt}^\epsilon$. Then
\begin{align*}
N\tau_{\min}\leq \sum_{i=1}^N t_{2,i}\leq \frac{\Delta H(x_0)}{v_{0}},
\end{align*}
which implies
\[
N\leq \left\lfloor \frac{\Delta H(x_0)}{v_{0}\tau_{\min}}\right\rfloor.
\]

Since $t_{1,i}\leq T_1$, the total time to reach $H_{\tgt}$, i.e. $\overline T$ s.t. 
\begin{align*}
\forall x_0\in S_0,\ \exists\, T\leq \overline{T}\ \mathrm{s.t.}\ \phi(T,x_0,\pi_{\mathcal{K}})\in H_{\tgt},
\end{align*}
satisfies 
\begin{align}
&\overline{T} \leq \sum_{i=1}^N t_{2,i}+NT_1\\
&\leq \frac{\Delta H(x_0)}{v_{0}} + \left\lfloor \frac{\Delta H(x_0)}{v_{0}\tau_{\min}}\right\rfloor T_1\\
&\leq \frac{L_{H}D_{\mathcal{X}}}{v_{0}}\Bigl(1+\frac{T_1}{\tau_{\min}}\Bigr).
\end{align}
Finally, by Assumption~\ref{ass_upper_bound_return_time}, the maximum time $T_{\max}$ to reach $S_{\tgt}$ from $S_0$ satisfies
\begin{align*}
T_{\max}\leq \overline{T}+T_2
\leq \frac{L_{H}D_{\mathcal{X}}}{v_{0}}\Bigl(1+\frac{T_1}{\tau_{\min}}\Bigr)+T_2.
\end{align*}
\end{proof}

\subsection{From Expert Demonstrations to NCPs}
\label{sec:alg}
This subsection explains how the assignment set is constructed from expert demonstrations and how it induces the nonparametric chain policy introduced in Section~\ref{sec:def_chain_policy}.

Consider a data set $\cD=\{(x_j,u_j(\cdot),T_{j})\}_{j=1}^M$ consisting of $M$ expert trajectories, where for each tuple $(x_j,u_j(\cdot),T_{j})\in\cD$, the trajectory $\phi(t,x_j,u_j)$, $t\in(0,T_{j}]$ satisfies,  $x_{j} \in S_0$ and $\phi(T_{j},x_j,u_j) \in S_{\tgt}^\delta, \forall j = 1,\dots,M$, where $S_\mathrm{tgt}^\delta \subseteq H_\mathrm{\tgt}^\epsilon\cap S_\mathrm{tgt}$. The assignment set is obtained by extracting local control snippets and their certified radii along this trajectory. As Algorithm~\ref{alg:offline_assignment_set} shows, starting from an anchor time $s\in[0,T_{j})$, define
$
x_i:=\phi(s,x_j,u_j).
$
Given a small $v_0$, for each candidate duration $t\in(0,T_{j}-s]$, let $u_{i,t}$ be the restriction of $u_j$ to $(s,s+t]$, so that $\phi(t,x_i,u_{i,t})=\phi(s+t,x_j,u_j)$.
Its certified radius is
\begin{align*}
\label{eq:radius_bound_chain_policy}
r_i(t)=
\frac{\Delta H(x_i)-\Delta H(\phi(t,x_i,u_{i,t}))-v_0t}
{L_H+L_H e^{Lt}}.
\end{align*}
If $r_i(t)>0$, then $u_{i,t}$ is valid on $\cB_{r_i(t)}(x_i)$. We choose $t_i\in\arg\max_{t\in(0,T_{j}-s],\,r_i(t)>0} r_i(t)$, set $\tau_i:=t_i$, $u_i:=u_{i,t_i}$, and $r_i:=r_i(t_i)$, and add $(x_i,r_i,u_i)$ to $\K$. Then let
$
\sigma_i:=\inf\{\delta\in(0,\tau_i]\mid \phi(s+\delta,x_j,u_j)\in\partial\cB_{r_i}(x_i)\},
$
with $\sigma_i:=\tau_i$ if the set is empty. The next anchor is
$
x_{i+1}:=\phi(s+\sigma_i,x_j,u_j),
$ which is displayed in Fig~\ref{fig:alg}.
Repeating this procedure until the trajectory reaches $S_{\tgt}$, and then over all demonstrations in $\cD$, yields
$
\K=\{(x_i,r_i,u_i)\}_{i=1}^N.
$ After the assignment set is constructed, the NCP follows Remark~\ref{rem:excution_of_NCP} for $\forall x \in \X$.

\begin{algorithm}[htbp]
\caption{Assignment-Set Construction}
\label{alg:offline_assignment_set}
\begin{algorithmic}[1]
\State \textbf{Input:} $\cD=\{(x_j,u_j(\cdot),T_{j})\}_{j=1}^M$
\State \textbf{Output:} $\K$
\State $\K\gets\emptyset$
\For{$j=1,\dots,M$}
    \State $s\gets 0$
    \While{$s<T_{j}$ and $\phi(s,x_j,u_j)\notin S_{\tgt}$}
        \State $x_i\gets \phi(s,x_j,u_j)$
        \State choose $t_i\in\arg\max_{t\in(0,T_{j}-s],\,r_i(t)>0} r_i(t)$
        \If{no such $t_i$ exists}
            \State \textbf{break}
        \EndIf
        \State $\tau_i\gets t_i$, $u_i\gets u_{i,t_i}$, $r_i\gets r_i(t_i)$
        \State $\K\gets\K\cup\{(x_i,r_i,u_i)\}$
        \State $\sigma_i\gets \inf\{\delta:\|\phi(s+\delta,x_j,u_j)-x_i\|=r_i\}$
        \If{the set is empty}
            \State $\sigma_i\gets \tau_i$
        \EndIf
        \State $s\gets s+\sigma_i$
    \EndWhile
\EndFor
\end{algorithmic}
\end{algorithm}

\begin{figure}
    \hspace{-2em}
    \centering
    \includegraphics[width=0.8\linewidth]{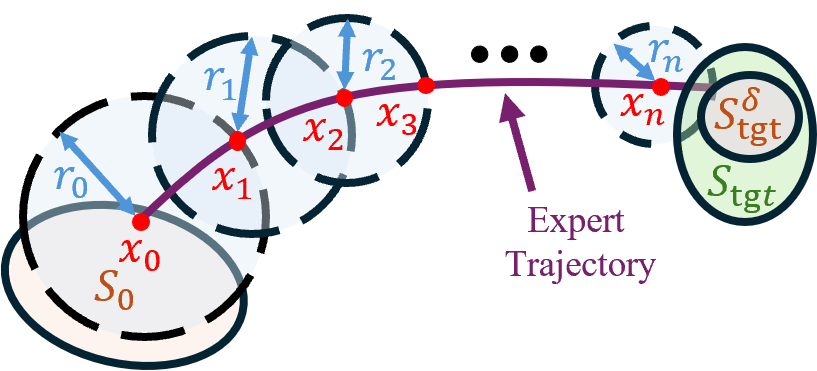}
    \caption{Assignment set construction.}
    \label{fig:alg}
\end{figure}

\section{Numerical Simulation}
\label{sec:numerical_simu}
We evaluate the proposed NCP on two systems: a spring-mass system and a pendulum, and compare them with a vanilla Behavior Cloning (BC) baseline. The state is written as $x=[q^\top,p^\top]^\top$, where $q$ and $p$ denote the generalized coordinates and momenta. The goal is to drive the state to a small neighborhood of a target state $x^*$, namely $S_{\tgt}=\{x:\|x-x^*\|\le \varepsilon\}$, where $\varepsilon = 0.1$. The BC policy is a three-layer multilayer perceptron with hidden sizes $(24,24,16)$, trained with Adam on the mean-squared imitation loss using learning rate $1.2\times 10^{-3}$, weight decay $5\times 10^{-4}$, and $40$ epochs. Expert demonstrations are generated offline by a nonlinear model predictive controller (NMPC) implemented in Python using CasADi and do-mpc.

For each system, we uniformly sample $500$ initial states from the set of states whose total energy is no greater than $\bar H$ for testing. For each number $M$ of expert demonstrations, we construct the proposed chain policy and train the BC baseline using the same $M$ trajectories. We report the success rate and the average reach time, where unsuccessful trajectories are assigned the simulation horizon. The horizon is $20\,\mathrm{s}$ for the spring-mass system and $150\,\mathrm{s}$ for the single pendulum. All experiments are conducted on a 3.2 GHz AMD Ryzen 7 7735HS CPU with 16 GB RAM.

\subsubsection{Spring-Mass}
The spring-mass dynamics are given by
\begin{align*}
x & =
\begin{bmatrix}
q\\
p
\end{bmatrix}
=
\begin{bmatrix}
q\\
m\dot{q}
\end{bmatrix},
H(x)=\frac{p^2}{2m}+\frac{k}{2}q^2,\\
\dot{x}
& =
\begin{bmatrix}
0 & 1\\
-1 & 0
\end{bmatrix}
\nabla H(x)
+
\begin{bmatrix}
0\\
1
\end{bmatrix}u.
\end{align*}
Here, $m=1$ is the mass, and $k=1$ is the spring stiffness. In the simulations, the control input is constrained to $u\in[-20,20]$ and $x^*=[0,0]^\top$.

\begin{figure}[htbp]
    \centering
    \hspace{-1em}
    \begin{subfigure}[b]{0.5\linewidth}
        \centering
        \includegraphics[width=\linewidth]{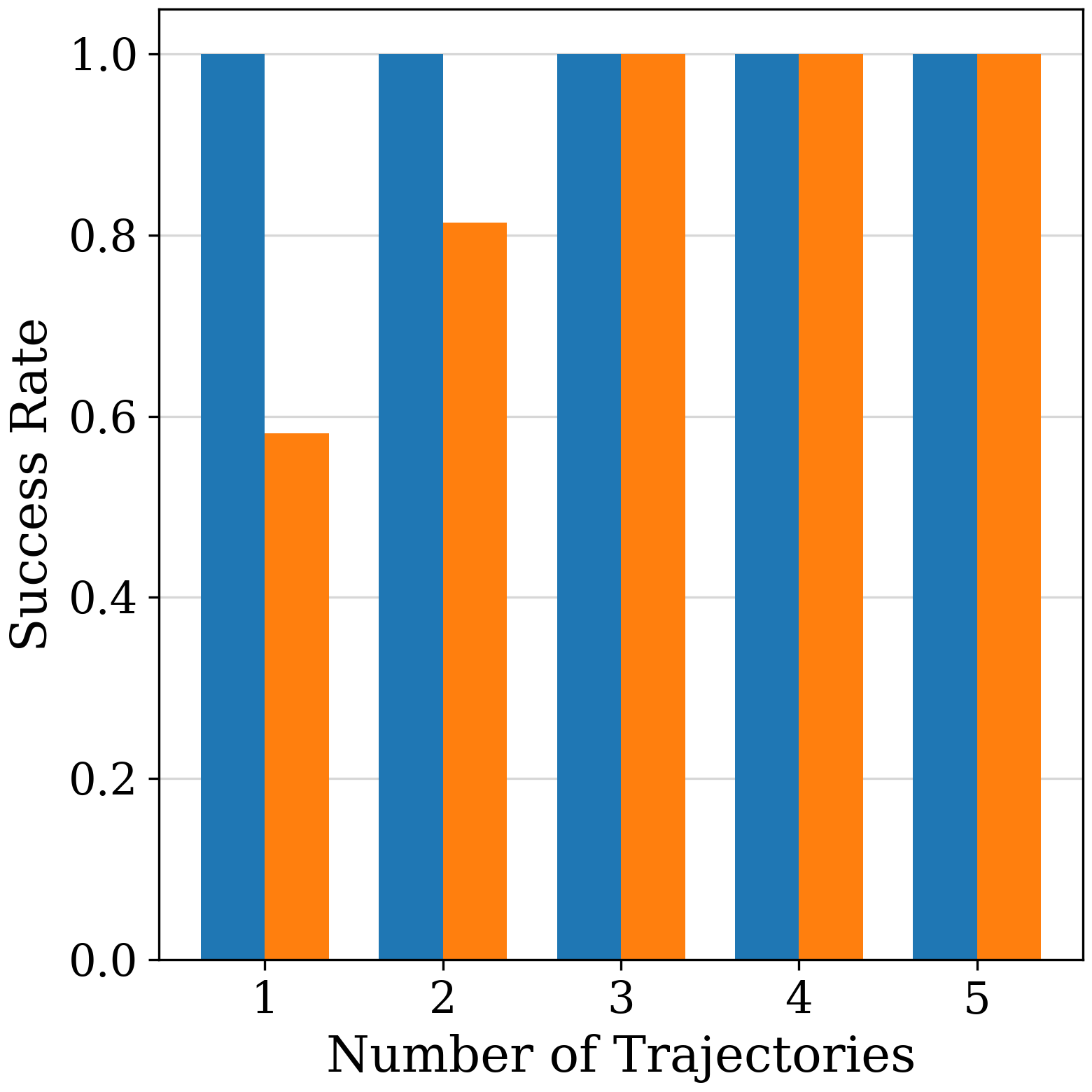}
        \caption{Success rate}
        \label{fig:spring_mass_success_rate}
    \end{subfigure}
    \hfill
    \begin{subfigure}[b]{0.5\linewidth}
        \centering
        \includegraphics[width=\linewidth]{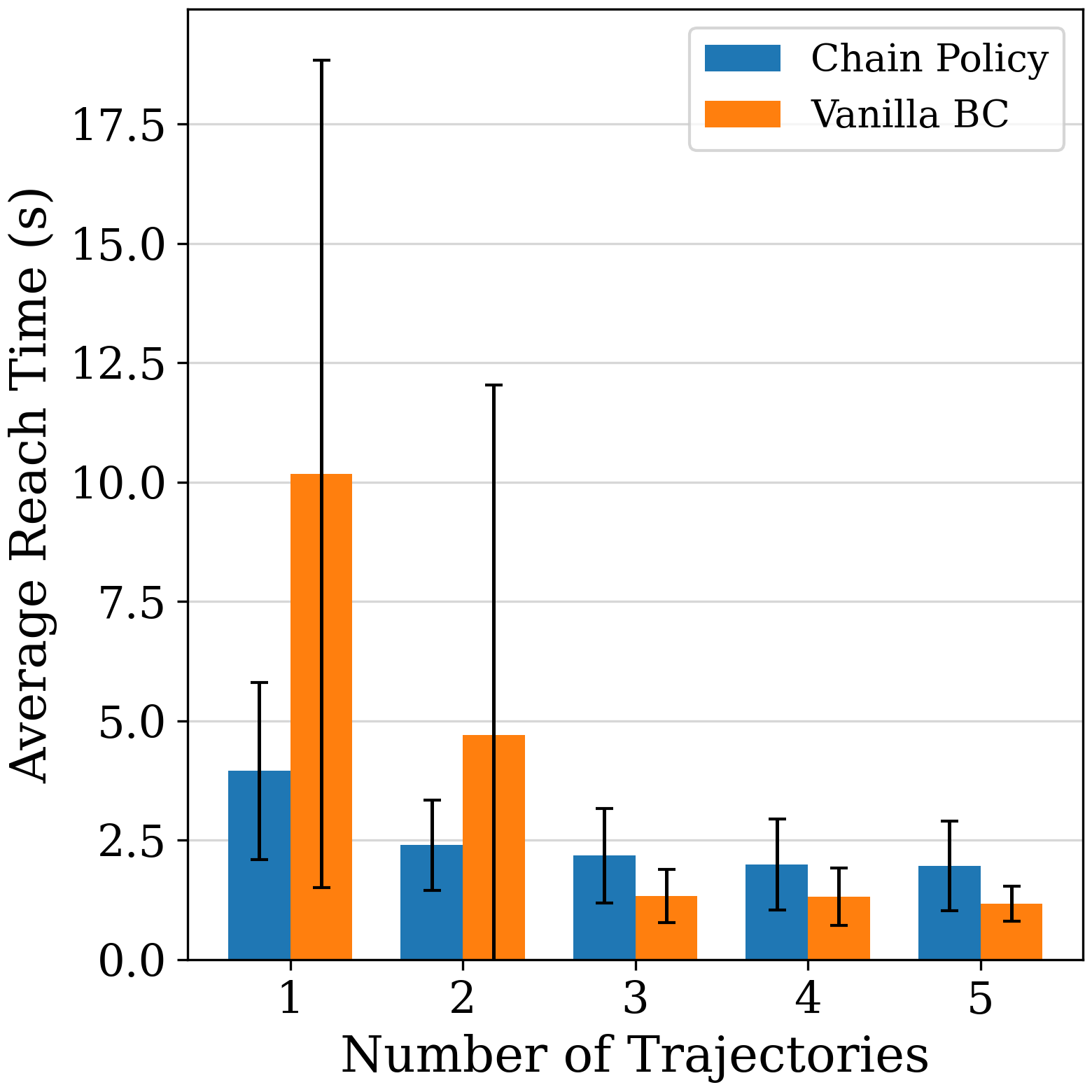}
        \caption{Average reach time}
        \label{fig:spring_mass_avg_time}
    \end{subfigure}
    \caption{Spring-mass system results. BC is averaged over seeds $5$--$9$. Reach-time error bars show one standard deviation over evaluated trajectories, with failures counted as $20\,\mathrm{s}$.}
    \label{fig:spring_mass_two_subfigs}
\end{figure}

As shown in Fig.~\ref{fig:spring_mass_success_rate}, the chain policy reaches the target from all $500$ test initial states for every $M=1,\ldots,5$. Averaged over training seeds $5,\ldots,9$, BC achieves success rates of $0.5812$ and $0.8136$ for $M=1$ and $M=2$, respectively, and $1.0$ for all five seeds when $M\geq3$. Its individual success rates at $M=1$ range from $0.212$ to $1.0$. Fig.~\ref{fig:spring_mass_avg_time} shows that the chain policy's mean reach time falls from $3.95\,\mathrm{s}$ at $M=1$ to $1.96\,\mathrm{s}$ at $M=5$, with the largest decrease between $M=1$ and $M=2$. Counting failures as $20\,\mathrm{s}$, BC's mean reach times are $10.18\,\mathrm{s}$ at $M=1$, $4.70\,\mathrm{s}$ at $M=2$, and approximately $1.17$--$1.32\,\mathrm{s}$ for $M\geq3$. Thus, the chain policy succeeds with one demonstration on this test set; BC is sensitive to its training seed with limited demonstrations and reaches the target faster on average for $M\geq3$.

\subsubsection{Single Pendulum}
The single pendulum dynamics are
\begin{align*}
x &=
\begin{bmatrix}
q\\
p
\end{bmatrix}
=
\begin{bmatrix}
\theta\\
m\ell^2\dot{\theta}
\end{bmatrix},
H(x)=\frac{p^2}{2m\ell^2}+mg\ell(1-\cos q),\\
\dot{x} &=
\begin{bmatrix}
0 & 1\\
-1 & 0
\end{bmatrix}
\nabla H(x)
+
\begin{bmatrix}
0\\
1
\end{bmatrix}u.
\end{align*}
Here, $m=1$ is the pendulum mass, $\ell=2$ is the pendulum length, and $g=9.81 \mathrm{m/s^2}$ is the gravitational acceleration. In the simulations, the control torque is constrained to $u\in[-20,20]$ and $x^*=[\pi,0]^\top$.

\begin{figure}[htbp]
    \centering
    \hspace{-1em}
    \begin{subfigure}[b]{0.5\linewidth}
        \centering
        \includegraphics[width=\linewidth]{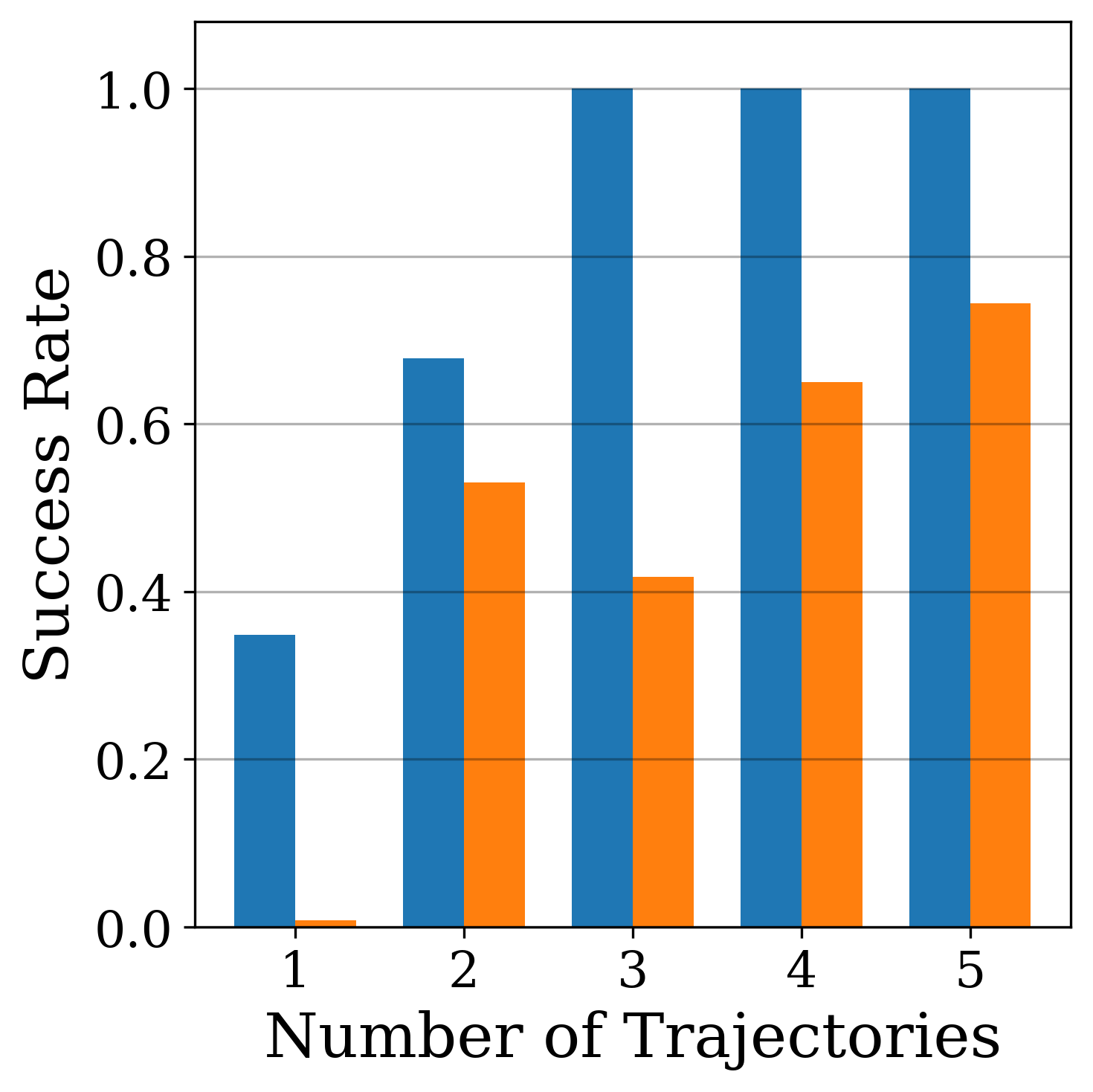}
        \caption{Success rate}
        \label{fig:single_pendulum_success_rate}
    \end{subfigure}
    \hfill
    \begin{subfigure}[b]{0.5\linewidth}
        \centering
        \includegraphics[width=\linewidth]{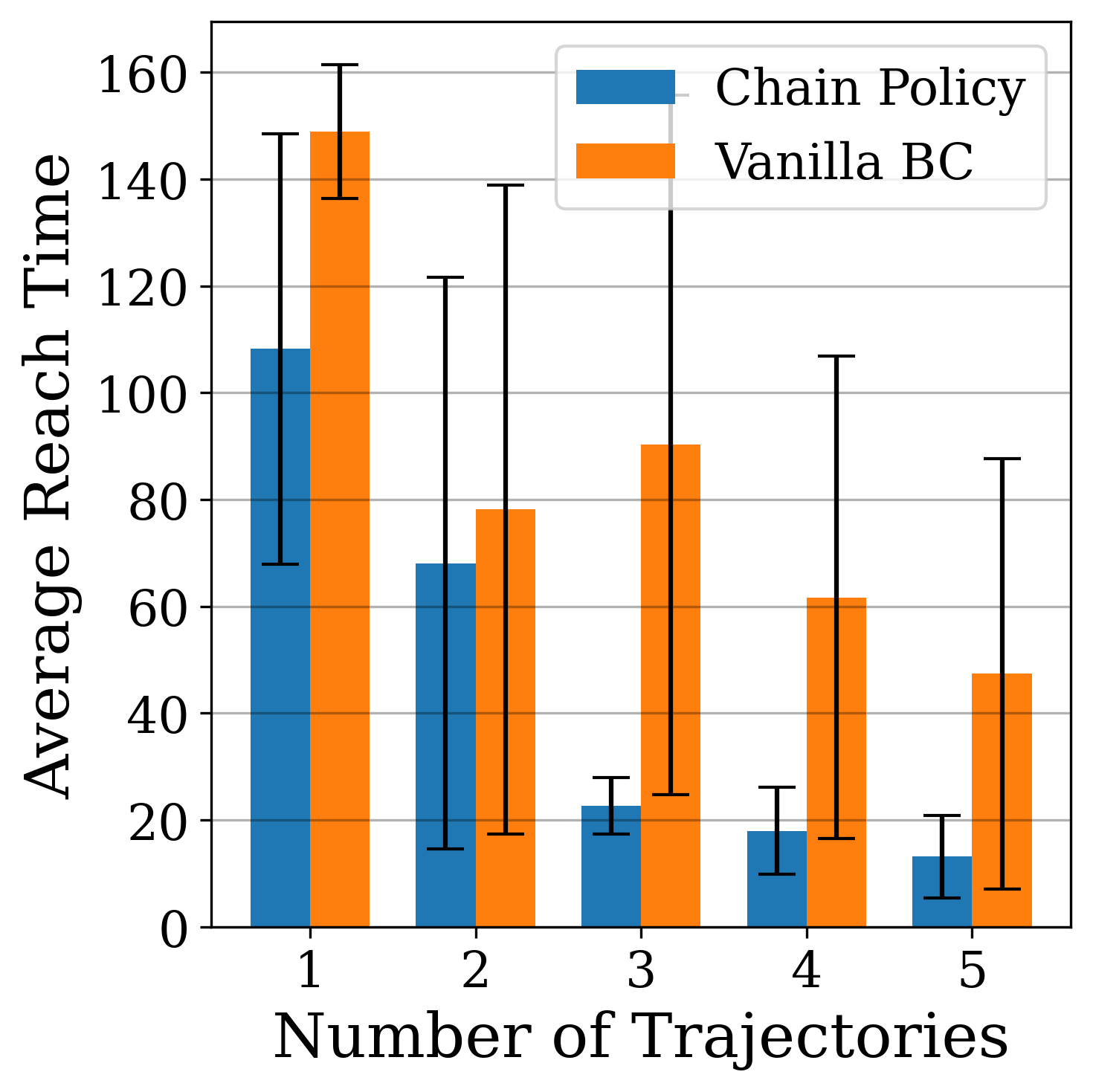}
        \caption{Average reach time}
        \label{fig:single_pendulum_avg_time}
    \end{subfigure}
    \caption{Single pendulum results.}
    \label{fig:single_pendulum_two_subfigs}
\end{figure}

In this task, as Fig~\ref{fig:single_pendulum_success_rate} shows, the success rate of the proposed chain policy increases from $0.348$ at $M=1$ to $0.678$ at $M=2$, and reaches $1.0$ for all $M\geq 3$. However, vanilla BC remains consistently inferior, with success rates $0.008$, $0.53$, $0.418$, $0.65$, and $0.744$ for $M=1,\dots,5$, respectively. Besides, the average reach time shown in Fig.~\ref{fig:single_pendulum_avg_time} also decreases significantly as the number of expert trajectories grows, from $114.71$s at $M=1$ to $13.16$s at $M=5$. This trend indicates that enlarging the assignment set with additional expert trajectories substantially improves the ability of the chain policy to find effective local control snippets and steer the system to the target set more efficiently. 

\section{Conclusions and Future Work}
\label{sec:conclusions_and_future_work}

This paper proposed a data-driven control framework for Hamiltonian systems that leverages physical structure as an inductive bias. Rather than learning policies over the entire state space, we exploit energy conservation and recurrence to construct control laws from a finite collection of locally verified trajectory segments. By combining controlled energy reduction with recurrence on invariant energy layers, we establish target-set reachability and derive finite-time guarantees under suitable conditions. Numerical examples demonstrate the effectiveness of the approach.

Future work will focus on extending these ideas to more general settings, including systems with multiple ergodic components and partial observability, as well as robust formulations that account for model uncertainty and energy dissipation~\cite{romero2024robust}. Another direction is to develop principled strategies for data selection and augmentation that improve coverage of the relevant low-dimensional structures while mitigating compounding errors~\cite{zhang2026action}.
\clearpage
\printbibliography

@inproceedings{chi2023diffusionpolicy,
  title={Diffusion Policy: Visuomotor Policy Learning via Action Diffusion},
  author={Chi, Cheng and Feng, Siyuan and Du, Yilun and Xu, Zhenjia and Cousineau, Eric and Burchfiel, Benjamin and Song, Shuran},
  booktitle={Proceedings of Robotics: Science and Systems},
  year={2023},
  doi={10.15607/RSS.2023.XIX.026}
}

@article{dai2020semi,
  title={A semi-algebraic optimization approach to data-driven control of continuous-time nonlinear systems},
  author={Dai, Ting and Sznaier, Mario},
  journal={IEEE Control Systems Letters},
  volume={5},
  number={2},
  pages={487--492},
  year={2020},
  publisher={IEEE}
}

@article{guo2021data,
  title={Data-driven stabilization of nonlinear polynomial systems with noisy data},
  author={Guo, Ming and De Persis, Claudio and Tesi, Pietro},
  journal={IEEE Transactions on Automatic Control},
  volume={67},
  number={8},
  pages={4210--4217},
  year={2021},
  publisher={IEEE}
}

@inproceedings{strasser2021data,
  title={Data-driven control of nonlinear systems: Beyond polynomial dynamics},
  author={Str{\"a}sser, Raphael and Berberich, Julian and Allg{\"o}wer, Frank},
  booktitle={IEEE Conference on Decision and Control (CDC)},
  pages={4344--4351},
  year={2021},
  organization={IEEE}
}

@article{monshizadeh2025versatile,
  title={A versatile framework for data-driven control of nonlinear systems},
  author={Monshizadeh, Nima and De Persis, Claudio and Tesi, Pietro},
  journal={IEEE Transactions on Automatic Control},
  year={2025},
  note={to appear}
}

@article{boffi2020learning,
  title={Learning stability certificates from data},
  author={Boffi, Nicholas M. and Tu, Stephen and Matni, Nikolai and Slotine, Jean-Jacques and Sindhwani, Vikas},
  journal={arXiv preprint arXiv:2008.05952},
  year={2020}
}

@article{oymak2018nonasymptotic,
  title={Non-asymptotic Identification of Linear Dynamical Systems from a Single Trajectory},
  author={Oymak, Samet and Ozay, Necmiye},
  journal={arXiv preprint arXiv:1806.05722},
  year={2018}
}

@article{zheng2021nonasymptotic,
  title={Non-Asymptotic Identification of Linear Dynamical Systems Using Multiple Trajectories},
  author={Zheng, Yang and Li, Na},
  journal={IEEE Control Systems Letters},
  volume={5},
  number={5},
  pages={1693--1698},
  year={2021},
  publisher={IEEE}
}

@article{dean2020sample,
  title={On the sample complexity of the linear quadratic regulator},
  author={Dean, Sarah and Mania, Horia and Matni, Nikolai and Recht, Benjamin and Tu, Stephen},
  journal={Foundations of Computational Mathematics},
  volume={20},
  number={4},
  pages={633--679},
  year={2020},
  publisher={Springer}
}

@article{depersis2019formulas,
  title={Formulas for data-driven control: Stabilization, optimality, and robustness},
  author={De Persis, Claudio and Tesi, Pietro},
  journal={IEEE Transactions on Automatic Control},
  volume={65},
  number={3},
  pages={909--924},
  year={2020},
  publisher={IEEE}
}

@inproceedings{coulson2019data,
  title={Data-enabled predictive control: In the shallows of the DeePC},
  author={Coulson, Jeremy and Lygeros, John and Dörfler, Florian},
  booktitle={European Control Conference (ECC)},
  pages={307--312},
  year={2019},
  organization={IEEE}
}

@article{berberich2020data,
  title={Data-driven model predictive control with stability and robustness guarantees},
  author={Berberich, Julian and Köhler, Johannes and Müller, Matthias A. and Allgöwer, Frank},
  journal={IEEE Transactions on Automatic Control},
  volume={66},
  number={4},
  pages={1702--1717},
  year={2021},
  publisher={IEEE}
}

@article{werner2024sample,
  title={On the sample complexity of stabilizing linear dynamical systems from data},
  author={Werner, Sebastian W. and Peherstorfer, Benjamin},
  journal={Foundations of Computational Mathematics},
  volume={24},
  number={3},
  pages={955--987},
  year={2024}
}

@inproceedings{hu2022sample,
  title={On the sample complexity of stabilizing LTI systems on a single trajectory},
  author={Hu, Yiheng and Wierman, Adam and Qu, Guannan},
  booktitle={Advances in Neural Information Processing Systems},
  volume={35},
  pages={16989--17002},
  year={2022}
}

@inproceedings{toso2025learning,
  title={Learning stabilizing policies via an unstable subspace representation},
  author={Toso, Leonardo F. and Ye, Lintao and Anderson, James},
  booktitle={IEEE Conference on Decision and Control (CDC)},
  pages={7543--7550},
  year={2025},
  organization={IEEE}
}

@book{vapnik1998statistical,
  title={Statistical Learning Theory},
  author={Vapnik, Vladimir N.},
  year={1998},
  publisher={Wiley}
}

@book{shalev2014understanding,
  title={Understanding Machine Learning: From Theory to Algorithms},
  author={Shalev-Shwartz, Shai and Ben-David, Shai},
  year={2014},
  publisher={Cambridge University Press}
}

@book{bishop2006pattern,
  title={Pattern Recognition and Machine Learning},
  author={Bishop, Christopher M.},
  year={2006},
  publisher={Springer}
}

@article{schulman2017proximal,
  title={Proximal policy optimization algorithms},
  author={Schulman, John and Wolski, Filip and Dhariwal, Prafulla and Radford, Alec and Klimov, Oleg},
  journal={arXiv preprint arXiv:1707.06347},
  year={2017}
}

@article{liu2025safety,
  title={Safety-Critical Control via Recurrent Tracking Functions},
  author={Liu, Jixian and Mallada, Enrique},
  journal={arXiv preprint arXiv:2510.01147},
  year={2025}
}

@inproceedings{liu2025recurrent,
  title={Recurrent control barrier functions: A path towards nonparametric safety verification},
  author={Liu, Jixian and Mallada, Enrique},
  booktitle={2025 IEEE 64th Conference on Decision and Control (CDC)},
  pages={7721--7727},
  year={2025},
  organization={IEEE}
}

@inproceedings{siegelmann2023recurrence,
  title={A recurrence-based direct method for stability analysis and gpu-based verification of non-monotonic lyapunov functions},
  author={Siegelmann, Roy and Shen, Yue and Paganini, Fernando and Mallada, Enrique},
  booktitle={2023 62nd IEEE Conference on Decision and Control (CDC)},
  pages={6665--6672},
  year={2023},
  organization={IEEE}
}

@article{sibai2026recurrence,
  title={Recurrence of nonlinear control systems: Entropy, bit rates, and finite alphabet controllers},
  author={Sibai, Hussein and Mallada, Enrique},
  journal={Nonlinear Analysis: Hybrid Systems},
  volume={59},
  pages={101649},
  year={2026},
  publisher={Elsevier}
}

@article{siegelmann2025data,
  title={Data-driven Practical Stabilization of Nonlinear Systems via Chain Policies: Sample Complexity and Incremental Learning},
  author={Siegelmann, Roy and Mallada, Enrique},
  journal={arXiv preprint arXiv:2510.03982},
  year={2025}
}

@inproceedings{zhang2026action,
  title={Action Chunking and Data Augmentation Yield Exponential Improvements in Behavior Cloning for Continuous Spaces},
  author={Zhang, Thomas TCK and Pfrommer, Daniel and Pan, Chaoyi and Matni, Nikolai and Simchowitz, Max},
  booktitle={International Conference on Learning Representations (ICLR)},
  year={2026},
  url={https://openreview.net/forum?id=jiWXDvw1Lf}
}

@article{romero2024robust,
  title={A robust adaptive velocity observer for mechanical systems transformed in cascade form},
  author={Romero, Jose Guadalupe},
  journal={Automatica},
  volume={165},
  pages={111671},
  year={2024},
  publisher={Elsevier}
}

@inproceedings{lew2022simple,
  title={A simple and efficient sampling-based algorithm for general reachability analysis},
  author={Lew, Thomas and Janson, Lucas and Bonalli, Riccardo and Pavone, Marco},
  booktitle={Learning for Dynamics and Control Conference},
  pages={1086--1099},
  year={2022},
  organization={PMLR}
}

@book{viana2016foundations,
  title={Foundations of ergodic theory},
  author={Viana, Marcelo and Oliveira, Krerley},
  number={151},
  year={2016},
  publisher={Cambridge University Press}
}

@incollection{bowen1975ergodic,
  title={The ergodic theory of Axiom A flows},
  author={Bowen, Rufus and Ruelle, David},
  booktitle={The theory of chaotic attractors},
  pages={55--76},
  year={1975},
  publisher={Springer}
}

@article{plante1972anosov,
  title={Anosov flows},
  author={Plante, Joseph F},
  journal={American Journal of Mathematics},
  volume={94},
  number={3},
  pages={729--754},
  year={1972},
  publisher={JSTOR}
}

@article{hopf1971ergodic,
  title={Ergodic theory and the geodesic flow on surfaces of constant negative curvature},
  author={Hopf, Eberhard},
  year={1971}
}
\appendices
\section{Additional Comparisons with Diffusion Policy and PPO}
\label{app:dp_ppo_comparison}

The code for reproducing these additional experiments is available at \url{https://github.com/JeanLew01/Nonparameric_Hamiltonian_Systems}.

We compare the proposed chain policy with Diffusion Policy (DP)~\cite{chi2023diffusionpolicy} and proximal policy optimization (PPO)~\cite{schulman2017proximal} on the spring-mass system and the single pendulum.

\subsection{Evaluation Protocol}
The physical parameters, input bounds, and targets are those of the two systems in Section~\ref{sec:numerical_simu}. Each method uses the same $500$ test states per system, sampled uniformly from $\{H\leq\bar H\}$ with seed $2026$, where $\bar H=2$ for spring-mass and $\bar H=160$ for the pendulum. The target radius is $0.1$, the integration step is $0.005\,\mathrm{s}$, and the action period is $0.02\,\mathrm{s}$. The evaluation horizons are $20\,\mathrm{s}$ and $150\,\mathrm{s}$, respectively. Unsuccessful trajectories contribute the full horizon to the mean reach time. Chain Policy is deterministic; DP and PPO use training seeds $0,1,2$, giving $1500$ evaluated trajectories per learned-method setting.

Chain Policy and DP use the same first $M$ NMPC demonstrations for $M=1,\ldots,5$. For spring-mass, the demonstration initial states are $(q,p)=(2,0),(-2,0),(0,-2),(0,2),(1.5,-1)$, in this order. For the pendulum, the first two demonstrations cover clockwise and counter-clockwise rotation; the remaining three are libration demonstrations with energy cap $2mg\ell$. Chain Policy constructs its assignment set separately for each $M$, using global Lipschitz bounds and $v_0=10^{-3}$, and executes certified snippets with zero input outside the support.

DP uses a state-conditioned diffusion model with a width-$256$ residual MLP and three FiLM-conditioned blocks. It predicts $16$ actions and executes the first $8$ before replanning. Training uses $100$ noise levels, $20000$ AdamW updates, batch size $256$, learning rate $10^{-3}$, and exponential-moving-average weights. Inference uses $10$ deterministic DDIM denoising steps with fresh seeded initial noise at each query. Angles are encoded as sine and cosine features.

PPO learns without expert demonstrations, using separate tanh actor and critic networks with hidden widths $(64,64)$. Adam starts at learning rate $3\times10^{-4}$; the clipping coefficient is $0.2$, the GAE parameter is $0.95$, and each rollout batch has $10$ optimization epochs. Observations and returns are normalized. The nominal interaction budgets are $10^6$ transitions for spring-mass and $10^7$ for the pendulum; complete rollout batches give $999424$ and $9961472$ transitions, respectively. Spring-mass training uses $64$ parallel environments, discount factor $0.99$, and a distance cost with a target-entry bonus. Pendulum training uses $1024$ environments, discount factor $0.995$, and potential-based shaping of distance and energy error with a target-entry bonus. Training initial states are sampled independently of the fixed test set. Evaluation applies the final policy's deterministic mean action, clipped to the input bounds.

\begin{figure}[htbp]
    \vspace{-1em}
    \hspace{-1em}
    \centering    \includegraphics[width=0.45\textwidth]{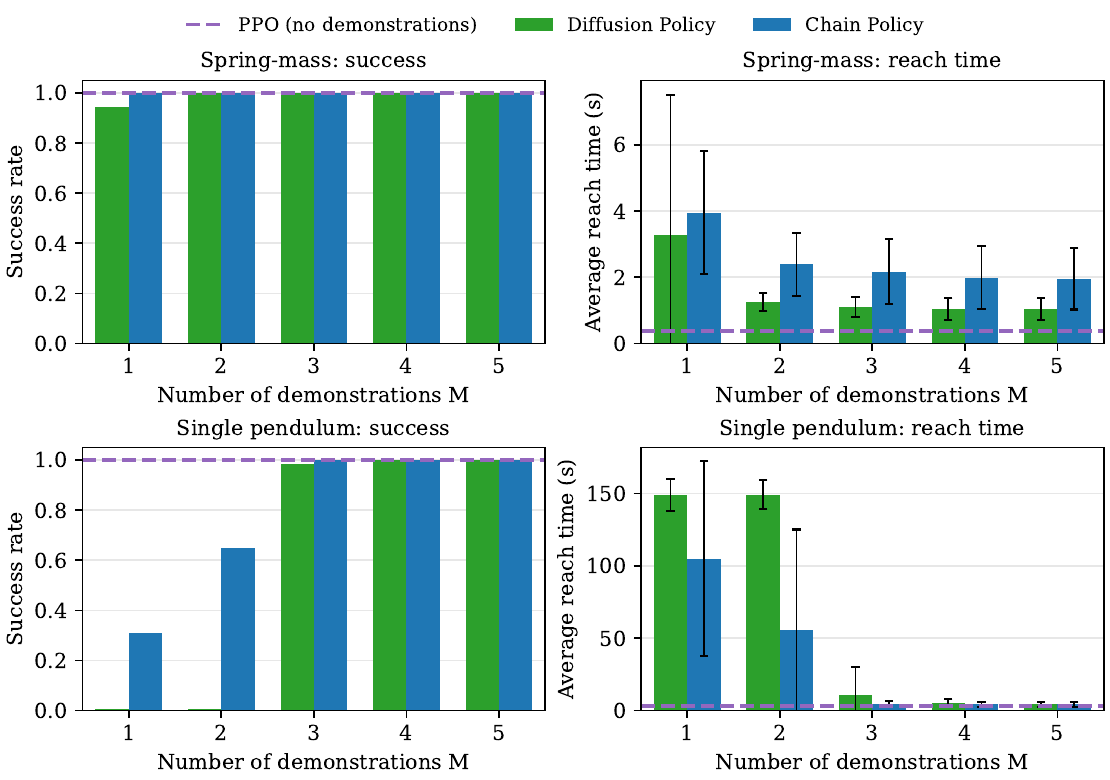}
    \caption{Target-reaching performance of Chain Policy, DP, and PPO. Chain Policy and DP use the same first $M$ expert demonstrations. DP and PPO results are pooled over three training seeds, each evaluated on the same $500$ test initial states. Reach-time error bars show one standard deviation, with failures counted as the full horizon. PPO is trained without demonstrations; its dashed line is independent of $M$.}
    \label{fig:additional_dp_ppo_comparison}
\end{figure}

\subsection{Success Rate and Reach Time}

Figure~\ref{fig:additional_dp_ppo_comparison} compares success rates and mean reach times on both systems. On spring-mass, Chain Policy succeeds from all $500$ states for every $M$, and its mean reach time decreases from $3.95\,\mathrm{s}$ at $M=1$ to $1.96\,\mathrm{s}$ at $M=5$. DP achieves $94.33\%$ success at $M=1$ and $100\%$ for $M\geq2$, with mean reach time decreasing from $3.28$ to $1.04\,\mathrm{s}$. PPO achieves $100\%$ success and a mean reach time of $0.38\,\mathrm{s}$ after approximately one million simulator transitions.

On the pendulum, Chain Policy's success rates are $30.8\%$, $64.8\%$, and $100\%$ for $M=1$, $M=2$, and all $M\geq3$, respectively. Its corresponding mean reach times are $105.03$, $55.36$, $4.52$, $4.38$, and $4.38\,\mathrm{s}$ for $M=1,\ldots,5$. DP attains success rates of $0.60\%$, $0.47\%$, $98.27\%$, $100\%$, and $100\%$, with mean reach times of $149.13$, $149.30$, $10.57$, $5.24$, and $4.18\,\mathrm{s}$. The improvement at $M=3$ coincides with adding the first libration demonstration to the two rotation demonstrations. Chain Policy reaches all test states with those three demonstrations, while DP still has a small failure fraction. PPO succeeds for all three seeds and reaches the target in $2.91\,\mathrm{s}$ on average after approximately ten million simulator transitions.

\subsection{Performance and Training Resources at \texorpdfstring{$M=3$}{M=3}}
\label{app:performance_resources}

\begin{table*}[htbp]
\centering
\small
\caption{Performance and training resources at $M=3$ for Chain Policy and DP, compared with PPO. DP and PPO use seeds $0,1,2$; Chain Policy is deterministic. Training time for learned methods is mean $\pm$ one standard deviation across seeds; expert generation and evaluation are excluded.}
\label{tab:additional_dp_ppo_resources}
\begin{tabular}{lllrcrr}
\toprule
System & Method & Training data & Device & Success (\%) & Reach time (s) & Training time (s) \\
\midrule
Spring-mass & Chain & 3 demonstrations & CPU & 100.00 & 2.17 & 4.62 \\
Spring-mass & DP & 3 demonstrations & GPU & 100.00 & 1.10 & $17.65\pm0.33$ \\
Spring-mass & PPO & $\approx10^6$ interactions & CPU & 100.00 & 0.38 & $87.44\pm2.00$ \\
\midrule
Single pendulum & Chain & 3 demonstrations & CPU & 100.00 & 4.52 & 12.39 \\
Single pendulum & DP & 3 demonstrations & GPU & 98.27 & 10.57 & $17.27\pm0.09$ \\
Single pendulum & PPO & $\approx10^7$ interactions & GPU & 100.00 & 2.91 & $147.08\pm3.65$ \\
\bottomrule
\end{tabular}
\end{table*}

Table~\ref{tab:additional_dp_ppo_resources} compares the three methods with three demonstrations for Chain Policy and DP. All methods achieve $100\%$ success on spring-mass. On the pendulum, Chain Policy and PPO achieve $100\%$ success, while DP achieves $98.27\%$. PPO has the shortest mean reach time on both systems, using approximately $10^6$ and $10^7$ simulator interactions, respectively.

\begin{figure}[htbp]
    \vspace{-1em}
    \hspace{-1em}
    \centering   \includegraphics[width=0.45\textwidth]{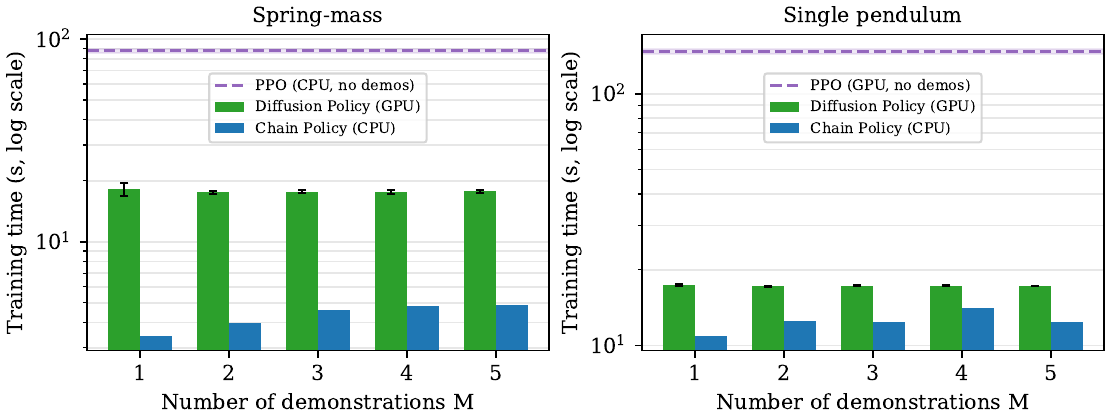}
    \caption{Training time for Chain Policy, DP, and PPO. Chain Policy includes Lipschitz-bound computation and assignment-set construction on the first $M$ demonstrations. DP bars show mean $\pm$ one standard deviation over three training seeds. PPO's dashed line and shaded band show its mean $\pm$ one standard deviation, independent of $M$. Expert generation and evaluation are excluded; PPO training includes simulator interaction. Devices are identified in the legends.}
    \label{fig:additional_training_comparison}
\end{figure}

Figure~\ref{fig:additional_training_comparison} reports the complete training-time sweep, and Table~\ref{tab:additional_dp_ppo_resources} gives the $M=3$ comparison. Chain Policy construction takes $3.42$--$4.89\,\mathrm{s}$ on spring-mass and $10.89$--$14.11\,\mathrm{s}$ on the pendulum. At $M=3$, its recorded training times are $4.62$ and $12.39\,\mathrm{s}$, respectively. DP takes approximately $17$--$18\,\mathrm{s}$ per seed across demonstration counts. PPO requires $87.44\pm2.00\,\mathrm{s}$ for spring-mass and $147.08\pm3.65\,\mathrm{s}$ for the pendulum.

Chain Policy runs on the CPU, with assignment construction parallelized across demonstrations. GPU training uses an NVIDIA GeForce RTX 4070 Laptop GPU. DP uses this GPU on both systems; PPO uses a CPU for spring-mass and this GPU for the pendulum. The wall-clock times therefore reflect both the learning algorithms and their device allocations. Expert demonstration generation and evaluation are excluded from training times. Chain Policy and DP use a fixed set of expert demonstrations, while PPO is trained with approximately one to ten million simulator transitions and task-specific reward shaping.

\end{document}